\documentclass[a4paper,11pt]{article}
\usepackage{amsfonts,amsmath,amssymb,amsthm,verbatim,placeins,xcolor,comment,placeins,caption,subcaption,enumerate,bm,graphicx}
\usepackage{geometry}
\usepackage[pdfauthor={DG},pdfstartview={FitH},colorlinks=true,citecolor=blue]{hyperref} %dvipdfm
\usepackage[english]{babel}
\usepackage[sort&compress]{natbib}
\usepackage{pgfplots}

\usepackage[utf8]{inputenc}
\DeclareUnicodeCharacter{0306}{}

\usepackage{soul}
\usepackage{cancel}

\newtheorem{theorem}{Theorem}[section]
\newtheorem{prop}{Proposition}[section]

\theoremstyle{remark}

\theoremstyle{definition}
\newtheorem{defi}{Definition}[section]

\theoremstyle{remark}
\newtheoremstyle{myremark}{}{}{\color{blue}\small}{}{\color{blue}\bfseries}{}{ }{}

\theoremstyle{myremark}

\newcommand{\E}{\mathbb{E}} %for expectation
\newcommand{\R}{{\mathbb R}}
\newcommand{\N}{{\mathbb N}}

\DeclareMathOperator{\Int}{int}
\DeclareMathOperator{\cl}{cl}
\newcommand{\mff}{\mathfrak{f}}

\usepackage{marginnote}

\begin{document}

\renewcommand*{\thefootnote}{\fnsymbol{footnote}}

\begin{center}
\Large{\textbf{Intermittency and multiscaling in limit theorems}}\\
%or\\
%Contrast between convergence in distribution and almost sure convergence using intermittency and multiscaling\\
\bigskip
\large{\today}\\
\bigskip
Danijel Grahovac$^1$\footnote{dgrahova@mathos.hr}, Nikolai N.~Leonenko$^2$\footnote{LeonenkoN@cardiff.ac.uk}, Murad S.~Taqqu$^3$\footnote{murad@bu.edu}\\
\end{center}

\bigskip
\begin{flushleft}
\footnotesize{
$^1$ Department of Mathematics, University of Osijek, Trg Ljudevita Gaja 6, 31000 Osijek, Croatia\\
$^2$ School of Mathematics, Cardiff University, Senghennydd Road, Cardiff, Wales, UK, CF24 4AG}\\
$^3$ Department of Mathematics and Statistics, Boston University, Boston, MA 02215, USA
\end{flushleft}

\bigskip

\textbf{Abstract: } It has been recently discovered that some random processes may satisfy limit theorems even though they exhibit intermittency, namely  an unusual growth of moments. In this paper we provide a deeper understanding of these intricate limiting phenomena. We show that intermittent processes may exhibit a multiscale behavior involving growth at different rates. To these rates correspond different scales.  In addition to a dominant scale, intermittent processes may exhibit secondary scales. The probability of these scales decreases to zero as a power function of time. For the analysis, we consider large deviations of the rate of growth of the processes. Our approach is quite general and covers different possible scenarios with special focus on the so-called supOU processes.

\medskip


\bigskip

\section{Introduction}

Limit theorems in probability have a long history and yet the understanding of general principles beyond the case of independent and identically distributed (i.i.d.) sequences is still far from complete. In the i.i.d. case, the type of limit depends typically only on the finiteness of moments of the underlying distribution. Recent results show that under dependence this simple characterization may break down. Namely, processes aggregated from finite variance sequences may converge to infinite variance processes and vice versa. See e.g.~\cite{pipiras2004slow,konstantopoulos1998macroscopic,leipus2003random,pilipauskaite2014joint,doukhan2019discrete,surgailis2004stable,sly2008nonstandard} for such examples.         

We were motivated   by results involving Ornstein-Uhlenbeck type processes ({supOU}).  SupOU processes are stationary processes for which the marginal distribution and the dependence structure can be modeled independently (see \cite{bn2001,barndorff2013levy,
barndorff2011multivariate,GLST2019Bernoulli,
barndorff2018ambit}). Since supOU processes are continuous time processes, one may easily aggregate them in order to obtain processes with stationary increments. It turns out that the four classes of processes can be obtained in the limit after suitable normalization, namely \cite{GLT2019Limit,GLT2019LimitInfVar}:
\begin{itemize}
    \item Brownian motion,
    \item fractional Brownian motion,
    \item a stable L\'evy process with  infinite variance and stationary independent increments,
     \item  a stable process with infinite variance and stationary dependent increments.
\end{itemize}
This convergence may be traced to a specific asymptotic behavior of moments, called \textit{intermittency} (see \cite{GLST2016JSP,GLST2019Bernoulli}). Such behavior resembles a similar one appearing in solutions of some stochastic partial differential equations (SPDE) (see e.g.~\cite{carmona1994parabolic,chen2015moments,gartner2007geometric,zel1987intermittency,khoshnevisan2014analysis,chong2018almost,chong2019intermittency}). In \cite{GLST2019Bernoulli,GLT2019Limit,GLT2019MomInfVar} a large class of integrated supOU processes has been showed to be intermittent.

The purpose of this paper is to provide a deeper understanding of the aforementioned limiting phenomena with the focus on supOU processes. Beyond limit theorems, one may investigate large deviations principles. We will show that the large deviations principle \textit{fails} to hold in its usual form for the integrated supOU processes with intermittency and long-range dependence. The large (or moderate) deviation statements provide bounds for the probabilities of the form
\begin{equation*}
P \left( |X(t)| > c b_t \right),
\end{equation*}
where $X$ is an aggregated process (partial sum or integrated process), $c>0$ and $\{b_t\}$ is a sequence of constants. One typically deals with processes for which such probabilities decay exponentially as $t\to \infty$. Hence, one considers
\begin{equation*}
\frac{1}{s_t} \log P \left( |X(t)| > c b_t \right),
\end{equation*}
in the limit as $t\to \infty$ for some sequence $\{s_t\}$ regularly varying at infinity (usually $s_t=t$). In contrast, for intermittent supOU processes, the probabilities of large deviations decay as a power function of $t$. To assess the rate of this decay we shall investigate
\begin{equation*}
\frac{1}{\log t} \log P \left( |X(t)| > c b_t \right).
\end{equation*}
To obtain such statements, we consider the large deviations not of the process ${X(t)}$ itself, but of the \textit{rate of growth} of the process 
$$
\log |X(t)|/\log t
$$
as $ t\rightarrow \infty$. The crucial point here is the observation that the rate function in such large deviations principle is the Legendre transform of the scaling function which measures the rate of growth of moments (see Sec.~\ref{sec3} for details).

We note that large and moderate deviations have been investigated in \cite{macci2017asymptotic} for the partial sums of a subclass of short-range dependent supOU processes satisfying the classical limit theorem with Brownian motion in the limit. Our results show, however, that the classical large deviation principle with exponentially decaying probabilities does not hold for the supOU processes with intermittency and, in particular, the results of \cite{macci2017asymptotic} cannot be extended to the intermittent case.

We show that intermittency may imply that the process may have different rates of growth, i.e.~it exhibits different scales. We will refer to such behavior as \textit{multiscaling} as this resembles the phenomena of multiscaling or separation of scales in physics literature (see e.g~\cite{carmona1994parabolic,zel1987intermittency,bertini1995stochastic,fujisaka1984theory,molchanov1991ideas,frisch1995turbulence}). Although these notions are widespread in physics, their presence has not been observed and properly described in the context of limit theorems in probability. Related phenomena also appear in the SPDE theory (see e.g.~\cite{khoshnevisan2014analysis,carmona1994parabolic,zel1987intermittency}).

Our results provide an interpretation of intermittency. Borrowing words from the monograph \cite[p.~84]{den2008large} who applied them to the parabolic Anderson model, we can view intermittency as a phenomenon where the dominant peaks of the process are localized on \textit{random islands} which occupy a fraction of the support that vanishes as time tends to infinity. Nevertheless, on these islands the peaks are so high that they determine the growth of the moments (see also \cite[p.~356]{bakry1994lectures}). Each higher moment is determined by a smaller fraction of the peaks. But, because we focus on the growth of moments, the use of the scaling function cannot be expected to reveal all the scales that the process may possess. We illustrate this by a simple example in Sec.~\ref{sec6}.

We conjecture that the multiscale phenomena are in the background of many peculiar limit theorems like the ones established in \cite{pipiras2004slow,konstantopoulos1998macroscopic,leipus2003random,pilipauskaite2014joint,doukhan2019discrete,surgailis2004stable,sly2008nonstandard}. Our results provide a general principle which enables investigating multiscale phenomena as soon as the asymptotic behavior of moments is available. For example, intermittency has also been confirmed in the so-called trawl processes \cite{grahovac2018intermittency}. Both trawl and supOU processes belong to the class of ambit processes (see \cite{barndorff2018ambit}) where more tractable examples can be expected.

The paper is organized as follows. In Sec.~\ref{sec2} we define intermittency and discuss its relation with limit theorems. In Sec.~\ref{sec3} we establish a general approach to large deviations of the rate of growth of the process. These results are then applied in Sec.~\ref{sec4} to several typical scenarios in limit theorems. In Sec.~\ref{subsec:34} we focus on the supOU processes. The results are illustrated by the simulations in Sec.~\ref{sec:simulations} and in Sec.~\ref{sec6} we provide some concluding remarks and a discussion.

\section{Intermittency}\label{sec2}

%=======================

Intermittency in the context of limit theorems has been introduced in \cite{GLST2016JSP,GLST2019Bernoulli} by adapting the similar notion from the theory of SPDEs (see e.g.~\cite{carmona1994parabolic,chen2015moments,gartner2007geometric,zel1987intermittency,khoshnevisan2014analysis,chong2018almost,chong2019intermittency}). Suppose that $X=\{X(t),\ t \geq 0\}$ is a process for which we would like to measure how fast its moments grow as $t\to \infty$. The \textit{scaling function} of $X$ at the point $q \in \R$ is
\begin{equation}\label{deftau}
\tau(q) = \tau_X(q) = \lim_{t\to \infty} \frac{\log \E |X(t)|^q}{\log t},
\end{equation}
where we assume the limit exists, possibly equal to $\infty$. If $\E|X(t)|^q=\infty$ for $t\geq t_0$, then $\tau(q)=\infty$. Note that $\tau(0)=0$ and that
\begin{equation*}
\frac{\tau(q)}{q} = \lim_{n\to \infty} \frac{\log  \left\lVert X(t) \right\rVert_q}{\log t},
\end{equation*}
where $\left\lVert X(t) \right\rVert_q= \left(  \E |X(t)|^q \right)^{1/q}$, which is the $L^q$ norm for $q\geq 1$. The following proposition lists some properties of $\tau$ and extends \cite[Prop.~2.1]{GLST2016JSP} to negative $q$ values. In \cite[Prop.~2.1]{GLST2016JSP} the assumption $\tau(q)\geq 0$ is missing in the statement that $\tau$ is nondecreasing.

\begin{prop}\label{propertiesoftau}
Suppose that $\tau$ is the scaling function of some process $X$ and let $$\mathcal{D}_{\tau}=\{q \in \R : \tau(q)<\infty\}. $$
 Then
\begin{enumerate}[(i)]
  \item $\tau$ is convex.
  \item $q \mapsto \tau(q)/q$ is nondecreasing on $\mathcal{D}_{\tau}$.
  \item If $\tau(q')\geq 0$ for some $q'>0$, then $\tau(q)\geq 0$ for every $q\geq q'$ and $\tau$ is nondecreasing on $\mathcal{D}_{\tau} \cap [q',\infty)$. In particular, if $\tau(q)\geq 0$ for any $q>0$, then $\tau$ is nondecreasing on $\mathcal{D}_{\tau} \cap [0,\infty)$.
  \item For any $q<0$, one has
  \begin{equation}\label{e:sfnegqbound}
  \tau(q) \geq q \inf_{q'>0} \frac{\tau(q')}{q'}.
  \end{equation}
  In particular, for any $q<0$ it holds that $\tau(q) \geq -\tau(-q)$.
\end{enumerate}
\end{prop}

\begin{proof}
\textit{(i)} Take $q_1, q_2 \in \R$ and $w_1, w_2 \geq 0$ such that $w_1+w_2=1$. By using H\"older's inequality we get
\begin{equation*}
\E|X(t)|^{w_1 q_1 + w_2 q_2} \leq \left( \E|X(t)|^{w_1 q_1 \frac{1}{w_1}} \right)^{w_1} \left( \E|X(t)|^{w_2 q_2 \frac{1}{w_2}} \right)^{w_2} =\left( \E|X(t)|^{q_1} \right)^{w_1} \left( \E|X(t)|^{q_2} \right)^{w_2}.
\end{equation*}
Taking logarithms, dividing by $\log t$ ($t>1$) and letting $t\to \infty$ yields $\tau(w_1 q_1 + w_2 q_2) \leq w_1 \tau(q_1) + w_2 \tau(q_2)$.

\textit{(ii)} For $q_1, q_2 \in \mathcal{D}_{\tau}$, $0< q_1 < q_2$, Jensen's inequality implies $\E|X(t)|^{q_1} = \E\left( |X(t)|^{q_2} \right)^{\frac{q_1}{q_2}} \leq  \left( \E |X(t)|^{q_2} \right)^{\frac{q_1}{q_2}}$ and hence $\frac{\E|X(t)|^{q_1}}{\log t}  \leq \frac{q_1}{q_2} \frac{\log \E |X(t)|^{q_2}}{\log t}$, which gives
\begin{equation}\label{tauprop1}
\tau(q_1) \leq \frac{q_1}{q_2} \tau(q_2) \iff \frac{\tau(q_1)}{q_1} \leq \frac{\tau(q_2)}{q_2}.
\end{equation}
If $q_1, q_2 \in \mathcal{D}_{\tau}$, $q_1 < q_2<0$, then we similarly obtain $\E|X(t)|^{q_2} = \E\left( |X(t)|^{q_1} \right)^{\frac{q_2}{q_1}} \leq  \left( \E |X(t)|^{q_1} \right)^{\frac{q_2}{q_1}}$, and $\tau(q_2) \leq \frac{q_2}{q_1} \tau(q_1) \iff \frac{\tau(q_1)}{q_1} \leq \frac{\tau(q_2)}{q_2}$. If $q_1, q_2 \in \mathcal{D}_{\tau}$, $q_1 < 0 < q_2$, then $\E|X(t)|^{q_1} = \E\left( |X(t)|^{q_2} \right)^{\frac{q_1}{q_2}} \geq  \left( \E |X(t)|^{q_2} \right)^{\frac{q_1}{q_2}}$, and
\begin{equation}\label{tauprop2}
\tau(q_1) \geq \frac{q_1}{q_2} \tau(q_2) \iff \frac{\tau(q_1)}{q_1} \leq \frac{\tau(q_2)}{q_2}.
\end{equation}

\textit{(iii)} If $\tau(q')\geq 0$, then taking $q_1=q'$ and $q_2=q$ in \eqref{tauprop1}, we have $\tau(q)\geq 0$. Now for arbitrary $q'<q_1<q_2$, \eqref{tauprop1} implies that $\tau(q_1) \leq \tau(q_2)$.

\textit{(iv)} This follows by taking $q_1=q$ and $q_2=q'$ in \eqref{tauprop2} and minimizing the right-hand side. That $\tau(q)\geq -\tau(-q)$ follows from \eqref{tauprop2} by putting $q_1=q$ and $q_2=-q$.
\end{proof}

In \cite{GLST2016JSP,GLST2019Bernoulli}, intermittency is defined by using the scaling function as follows.

\begin{defi}\label{def:int}
A stochastic process $\{X(t), t \geq 0\}$ with the scaling function $\tau$ is \emph{intermittent} if there exist $p, r \in \mathcal{D}_\tau$, $p<r$ such that
\begin{equation}\label{e:inter}
    \frac{\tau(p)}{p} < \frac{\tau(r)}{r}.
\end{equation}
\end{defi}

Since $q \mapsto \tau(q)/q$ is always non-decreasing, intermittency refers to a situation when there are points of strict increase in this mapping. In particular, $L^q$ norms of the process may grow at different rates for different $q$.

Suppose now that $\{X(t), \, t \geq 0\}$ is a \textit{self-similar process} with self-similarity parameter $H$, that is, the finite-dimensional distributions of $\{X(ct)\}$ are the same as those of $\{c^H X(t)\}$. The scaling function of $X$ is then $\tau(q)=Hq$ for $q \in \mathcal{D}_\tau$ and therefore a self-similar process cannot be intermittent.

To see how intermittency is related to limit theorems, note that by Lamperti's theorem \cite[Thm.~2.8.5]{pipiras2017long} every process satisfying limit theorem is asymptotically self-similar. More precisely, let $X=\{X(t), \, t \geq 0\}$ and $Z=\{Z(t), \, t \geq 0\}$ be two processes such that $Z(t)$ is nondegenerate for every $t>0$ and suppose that for a sequence $\{a_T\}$, $a_T>0$, $\lim_{T\to \infty} a_T = \infty$, one has
\begin{equation}
    \left\{ \frac{X(Tt)}{a_T} \right\} \overset{fdd}{\to} \left\{ Z(t) \right\}, \label{limittheorem}
\end{equation}
with convergence in the sense of convergence of all finite dimensional distributions as $T \to \infty$. By Lamperti's theorem, $Z$ is $H$-self-similar for some $H>0$. If in \eqref{limittheorem} there is also convergence of moments, then the scaling function of $X$ would be the same as the scaling function of the limit $Z$ \cite[Thm.~1]{GLST2019Bernoulli}. Hence, $\tau(q)=Hq$ for every $q$ such that
\begin{equation}
    \frac{\E| X(Tt)|^q}{a_T^q} \to \E |Z(t)|^q, \quad \forall t \geq 0. \label{thm:limitmom}
\end{equation}
For intermittent processes satisfying a limit theorem in the sense of \eqref{limittheorem},  convergence of moments as in \eqref{thm:limitmom} must fail to hold for some range of $q$. The next proposition shows that convergence of moments in \eqref{limittheorem} typically holds for moments of order $q$ in some neighborhood of the origin. In this range of $q$ the scaling function of $X$ must then be of the form $\tau(q)=Hq$.
%\marginnote{\textcolor{red}{One question is whether $\tau$ must always be linear in some neighborhood of the origin or, equivalently, whether \eqref {limittheorem} implies that \eqref{thm:limitmom} holds at least for small $q$? In general no, consider a sequence $2^n \1_{[0,1/n]}$ which converges to zero but all positive order moments diverge.}}

\begin{prop}
Suppose now that \eqref{limittheorem} holds for some nondegenerate processes $X$ and $Z$ and let $\mathcal{Q}$ be a set of $q \in \R$ for which \eqref{thm:limitmom} holds.
\begin{enumerate}[(i)]
\item If $0<r<s$, $\E |Z(1)|^s< \infty$ and $s\in \mathcal{Q}$, then $r \in \mathcal{Q}$.
\item If $s<r<0$, $\E |Z(1)|^s< \infty$ and $s\in \mathcal{Q}$, then $r \in \mathcal{Q}$.
\end{enumerate}
\end{prop}

\begin{proof}
\textit{(i)} By \cite[Thm.~5.4]{billingsley1968} we have that  $\left\{ \left| X(Tt)/a_T\right|^s \right\}$ is uniformly integrable, hence $\sup_n \E \left| X(Tt)/a_T\right|^s < \infty$. For $\varepsilon>0$ such that $0<r<r+\varepsilon<s$, we have by Jensen's inequality
\begin{equation*}%\label{e:Jensen:posq}
\E|X(t)|^{r+\varepsilon} = \E\left( |X(t)|^{s} \right)^{\frac{r+\varepsilon}{s}} \leq  \left( \E |X(t)|^{s} \right)^{\frac{r+\varepsilon}{s}}.
\end{equation*}
It follows that $\sup_n \E \left| X(Tt)/a_T\right|^{r+\varepsilon} < \infty$ and thus $\{\left| X(Tt)/a_T\right|^r \}$ is uniformly integrable and $r \in \mathcal{Q}$. \textit{(ii)} follows similarly. %The proof is the same as the proof of \textit{(i)} except that we use \eqref{e:Jensen:negq} instead of \eqref{e:Jensen:posq}.
\end{proof}

\section{Rate of growth of the process}\label{sec3}

We shall investigate the \textit{rate of growth} of the process $\{X(t)\}$ by considering
\begin{equation}\label{e:RY}
R_X(t) = \frac{\log |X(t)|}{\log t},
\end{equation}
as $t\to \infty$. We implicitly assume $X(t)$ does not have a probability point mass at zero, hence $\log |X(t)|<\infty$ a.s. Note that this definition of the rate of growth is tailored at processes that grow roughly as a power function of time. If one would be interested in processes growing exponentially in time (e.g.~solutions of some SPDEs), then $\log |X(t)|/t$ could be investigated. In the context of limit theorems, $\{X(t)\}$ will be a partial sum process or an integrated process.

\begin{prop}\label{prop:logYlimit}
Suppose that $\{X(t), \ t \geq 0\}$ satisfies \eqref{limittheorem} for some process $Z$ and a sequence of constants $\{a_T\}$. Then for some $H>0$ and for every $t>0$
\begin{equation*}
    \frac{\log |X(Tt)|}{\log T} \overset{P}{\to} H, \quad \text{ as } T \to \infty.
\end{equation*}
\end{prop}

\begin{proof}
By Lamperti's theorem \cite[Thm.~2.8.5]{pipiras2017long}, $Z$ is $H$-self-similar and $a_T=T^H L(T)$ for some $H>0$ and $L$ slowly varying at infinity. By the continuous mapping theorem we have that
\begin{equation*}
    \log T \left( \frac{\log |X(T)|}{\log T} - \frac{\log a_T}{\log T} \right) \overset{d}{\to} \log |Z(1)|,
\end{equation*}
so that
\begin{equation*}
    \frac{\log |X(T)|}{\log T} - \frac{\log a_T}{\log T} \overset{P}{\to} 0,
\end{equation*}
which proves the statement since
\begin{equation*}
\lim_{T\to \infty} \frac{\log a_T}{\log T}= \lim_{T\to \infty} \log \frac{H \log T + \log L(T)}{\log T} = H.
\end{equation*}
\end{proof}

Proposition \ref{prop:logYlimit} shows that for processes satisfying the limit theorem in the classical sense, the rate of growth $\log |X(T)|/\log T$ converges in probability to the self-similarity parameter $H$ of the limiting process. Roughly speaking, this means that $X(T)$ is typically of the order $T^H$.

However, besides the dominant \textit{scale} $t^H$, the limit theorem itself does not tell us if $X(t)$ exhibits any other scales which may be of the larger order but with probability decaying to zero. These scales may be identified by investigating the large deviations of the rate of growth. For this we use G\"artner-Ellis theorem which we recall here in a slightly more general version than \cite[Thm.~2.3.6]{dembo1998large} allowing for general speed $s_t$ and uncountable family of measure (see Remark (a) on p.~44 of \cite{dembo1998large}; see also \cite{ellis1984large}).

First we recall some related notions. For the function $f$ on the real line we denote by $f^*$ its \textit{Legendre(-Fenchel) transform}:
\begin{equation}\label{e:LFtransform}
f^* (x) = \sup_{q \in \R} \left\{ q x - f(q) \right\}
\end{equation}
and by $\mathcal{D}_{f}$ the set
\begin{equation*}
\mathcal{D}_{f}=\left\{ q \in \R : f(q)<\infty \right\}.
\end{equation*}
The point $x\in \R$ is an \textit{exposed point} of $f^*$ if for some $\lambda \in \R$ and all $y \neq x$
\begin{equation*}
f^*(y) - f^*(x) > \lambda (y-x).
\end{equation*}
The real number $\lambda$ is called an \textit{exposing hyperplane}.

\begin{theorem}\label{thm:GE}
Suppose $\{R(t), \, t \geq 0\}$ is a family of random variables with $R(t)$ having distribution $\mu_t$, $\{s_t\}$ is a sequence of positive numbers, $s_t \to \infty$, and define
\begin{equation}\label{e:lambdan}
\Lambda_t (q) = \log \E \left[ e^{q R(t)}\right].
\end{equation}
Assume that for each fixed $q \in \R$ the limit
\begin{equation}\label{e:lambda}
\Lambda(q) = \lim_{t\to \infty} \frac{1}{s_t} \Lambda_t( s_t q)
\end{equation}
exists as an extended real number and assume that zero belongs to the interior of the set $\mathcal{D}_\Lambda = \left\{ q \in \R : \Lambda(q)<\infty \right\}$. Then:
\begin{enumerate}[(i)]
\item For any closed set $C$, the following upper large deviation bound holds:
\begin{equation*}
\limsup_{t\to \infty} \frac{1}{s_t} \log \mu_t(C) \leq - \inf_{x \in C} \Lambda^* (x).
\end{equation*}
\item For any open set $O$, it holds that
\begin{equation}\label{e:lowerboundGEthm}
\liminf_{t\to \infty} \frac{1}{s_t} \log \mu_t(O) \geq - \inf_{x \in O \cap E } \Lambda^* (x),
\end{equation}
where $E$ is the set of exposed points of $\Lambda^*$ whose exposing hyperplanes belong to $\Int (\mathcal{D}_{\Lambda})$.
\end{enumerate}
\end{theorem}

If $O\cap E$ in \eqref{e:lowerboundGEthm} may be replaced with $O$, then one says the large deviation principle holds for $\{\mu_t\}$ with \textit{speed} $\{s_t\}$ and (good) rate function $\Lambda^*$. This happens if $\Lambda$ is essentially smooth and a lower semicontinuous function (see \cite{dembo1998large} for details).

We now return to the rate of growth \eqref{e:RY} of some process $\{X(t)\}$. As we will see from the examples below, apart from the dominant rate of growth,  intermittent processes may also exhibit other rates of growth and the probability of observing these rates decays as a power function of $t$ as $t\to \infty$. For this reason we choose $s_t=\log t$ in the G\"artner-Ellis theorem. We will then be able to identify large deviation probabilities that are power function of $t$. For \eqref{e:lambdan} we get
\begin{equation*}
\Lambda_t(q) = \log \E \left[\exp \left\{ q \frac{\log |X(t)|}{\log t}\right\} \right]
\end{equation*}
and \eqref{e:lambda} equals
\begin{equation*}
\Lambda(q) = \lim_{t\to \infty}  \frac{1}{\log t} \log \E \left[\exp \left\{q \log |X(t)| \right\} \right] = \lim_{t\to \infty}  \frac{1}{\log t} \log \E  |X(t)|^q = \tau(q),
\end{equation*}
provided the limit exists as an extended real number. Hence, \textit{the scaling function \eqref{deftau} plays the role of the function $\Lambda$ in the G\"artner-Ellis theorem for the rate of growth }and $\Lambda^*=\tau^*$ is the Legendre transform of the scaling function. From Thm.~\ref{thm:GE} we get the following.

\begin{theorem}\label{thm:mainGEapp}
Suppose $\{X(t), \ t \geq 0\}$ is a process with the scaling function $\tau$ such that for any $q\in \R$ the limit in \eqref{deftau} exists as an extended real number and $0 \in \Int (\mathcal{D}_\tau)$. Then:
\begin{enumerate}[(i)]
\item For any closed set $C$,
\begin{equation*}
\limsup_{t\to \infty} \frac{1}{\log t} \log P \left( \frac{\log |X(t)|}{\log t} \in C \right) \leq - \inf_{x \in C} \tau^* (x).
\end{equation*}
\item For any open set $O$,
\begin{equation*}%\label{e:lowerboundGEthm}
\liminf_{t\to \infty} \frac{1}{\log t} \log P \left( \frac{\log |X(t)|}{\log t} \in O \right) \geq - \inf_{x \in O \cap E } \tau^* (x),
\end{equation*}
where $E$ is the set of exposed points of $\tau^*$ whose exposing hyperplane belongs to $\Int (\mathcal{D}_{\tau})$.
\end{enumerate}
\end{theorem}

For the process $X=\{X(t), \ t \geq 0\}$, let
\begin{equation}\label{e:qunderover}
\begin{aligned}
\overline{q}(X) &= \sup \{ q >0 : \E|X(t)|^q < \infty  \ \forall t\},\\
\underline{q}(X) &= \inf \{ q <0 : \E|X(t)|^q < \infty  \ \forall t\}.
\end{aligned}
\end{equation}
To apply Thm.~\ref{thm:mainGEapp}, zero must be in the interior of $\mathcal{D}_\tau=\{q \in \R : \tau(q) < \infty\}$. A necessary condition for this to hold is that $\overline{q}(X)>0$ and $\underline{q}(X)<0$. If the limit in \eqref{deftau} is finite, the moments are finite in the range $(\underline{q}(X),\overline{q}(X))$.

We may also state \textit{(i)} and \textit{(ii)} of Thm.~\ref{thm:mainGEapp} equivalently as
\begin{equation}\label{e:LDPsetA}
\begin{aligned}
- \inf_{x \in \Int(A) \cap E } \tau^* (x) &\leq \liminf_{t\to \infty} \frac{1}{\log t} \log P \left( \frac{\log |X(t)|}{\log t} \in A \right)\\
&\leq \limsup_{t\to \infty} \frac{1}{\log t} \log P \left( \frac{\log |X(t)|}{\log t} \in A \right) \leq - \inf_{x \in \cl (A)} \tau^* (x),
\end{aligned}
\end{equation}
where $\cl(A)$ denotes the closure of a Borel set $A\subset \R$. In Sec.~\ref{sec4} and Sec.~\ref{subsec:34} we will illustrate many applications of Thm.~\ref{thm:mainGEapp}, but the general principle is the following. Suppose that we are interested in the scale $t^s$ (rate $s$). In an ideal situation, putting $A=(s-\varepsilon,s+\varepsilon)$ in \eqref{e:LDPsetA} would enable describing the probability of observing the scale $t^s$ since \eqref{e:LDPsetA} roughly tells us that
\begin{equation*}
t^{- \inf_{x \in (s-\varepsilon,s+\varepsilon) \cap E } \tau^* (x)} \lesssim P \left( t^{s-\varepsilon} \leq |X(t)| \leq t^{s+\varepsilon} \right) \lesssim t^{- \inf_{x \in [s-\varepsilon,s+\varepsilon]} \tau^* (x)}.
\end{equation*}
Hence, in the plot of $\tau^*(x)$ for a range of $x$ values, we may interpret $x$ as the rates of growth and $\tau^*(x)$ as the rate of decay of the probability of observing these rate. We illustrate this reasoning in Fig.~\ref{fig:idea}.

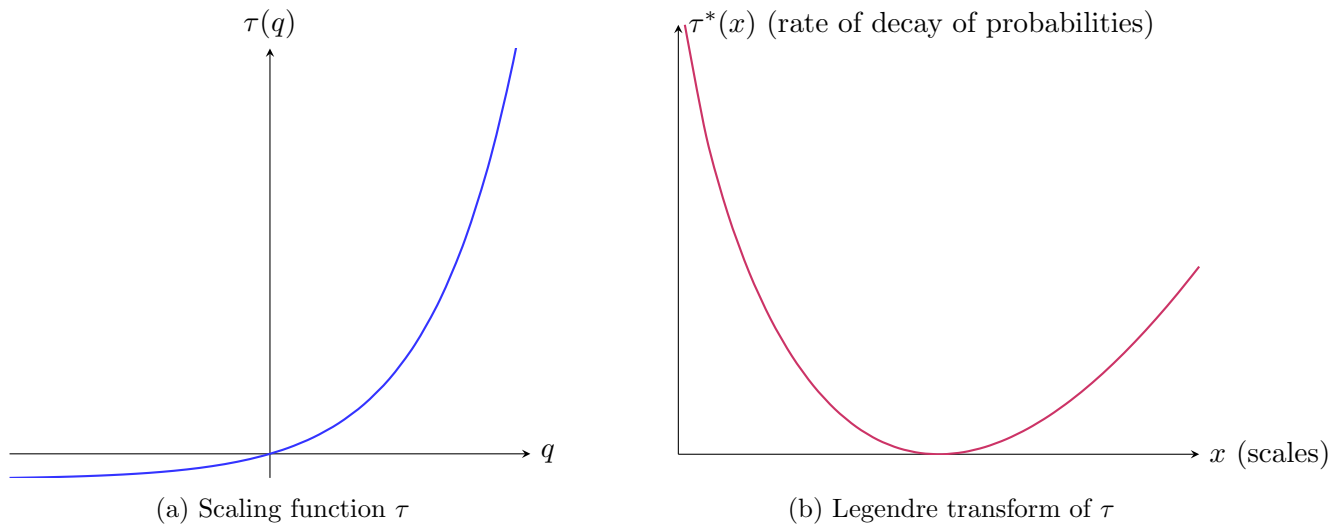
\begin{figure}[h!]
\centering
\begin{subfigure}[b]{0.45\textwidth}
\begin{tikzpicture}[domain=-2:2]
\begin{axis}[
axis lines=middle,
xlabel=$q$, xlabel style={at=(current axis.right of origin), anchor=west},
ylabel=$\tau(q)$, ylabel style={at=(current axis.above origin), anchor=south},
xtick={0},
xticklabels={$0$},
xmin=-3,
xmax=3,
ymajorticks=false
]
\addplot[smooth,thick,white!20!blue,domain=-3:4]{e^x-1} node [pos=0.5,left]{$Hq\ $};
\end{axis}
\end{tikzpicture}
\caption{Scaling function $\tau$}
\label{fig:idea2}
\end{subfigure}
\hfill
\begin{subfigure}[b]{0.45\textwidth}
\begin{tikzpicture}[domain=0:2]
\begin{axis}[
axis lines=middle,
xlabel=$x$ (scales), xlabel style={at=(current axis.right of origin), anchor=west},
ylabel=$\tau^*(x)$ (rate of decay of probabilities), ylabel style={at=(current axis.above origin), anchor=west},
xmin=0,
xmax=2,
xmajorticks=false,
ymajorticks=false
]
\addplot[smooth,thick,white!20!purple,domain=0.025:2]{1-x+x*ln(x)};
%\node (source) at (axis cs:-1,0.2){};
%\node (destination) at (axis cs:1,0.2){};
%\draw[->, thick](source) to [out=180,in=90] (destination);
\end{axis}
\end{tikzpicture}
\caption{Legendre transform of $\tau$}
\label{fig:idea2}
\end{subfigure}
\caption{From the scaling function to rate of growth}
\label{fig:idea}
\end{figure}

\section{Rate of growth in different scenarios}\label{sec4}

In this section we investigate the rate of growth of process by using the results of Sec.~\ref{sec3}. We consider different scenarios related to limit theorems.

\subsection{Limit theorems where all moments are finite and converge}\label{subsec:31}
Suppose that $X$ satisfies the limit theorem as in \eqref{limittheorem}, that all the moments are finite and converge, that is
\begin{equation}\label{e:mooconvsubsec41}
    \frac{\E| X(Tt)|^q}{a_T^q} \to \E |Z(t)|^q,
\end{equation}
for every $q\in \R$ and $t\geq 0$. The scaling function \eqref{deftau} of $X$ is \cite[Thm.~1]{GLST2019Bernoulli}
\begin{equation*}
    \tau(q)= H q, \quad q \in \R.
\end{equation*}
See Fig.~\ref{fig:subsec:31tau}. In this case, $X$ is not intermittent. The Legendre transform of $\tau$ is
\begin{equation}\label{e:ext}
\tau^*(x) = \sup_{q \in \R} \left\{ q x - \tau(q) \right\} = \sup_{q \in \R} \left\{ q (x - H) \right\}= \begin{cases}
0, & \text{ if } x=H,\\
\infty, & \text{ otherwise},
\end{cases}
\end{equation}
and the set of exposed points of $\tau^*$ is $E=\{H\}$ (see Fig.~\ref{fig:subsec:31tau*}).

\begin{figure}[h!]
\centering
\begin{subfigure}[b]{0.45\textwidth}
\begin{tikzpicture}[domain=-2:2]
\begin{axis}[
axis lines=middle,
xlabel=$q$, xlabel style={at=(current axis.right of origin), anchor=west},
ylabel=$\tau(q)$, ylabel style={at=(current axis.above origin), anchor=south},
xtick={0},
xticklabels={$0$},
xmin=-2,
xmax=4,
ymajorticks=false
]
\addplot[thick,white!20!blue,domain=-2:4]{0.625*x} node [pos=0.5,left]{$Hq\ $};
\end{axis}
\end{tikzpicture}
\caption{Scaling function $\tau$}
\label{fig:subsec:31tau}
\end{subfigure}
\hfill
\begin{subfigure}[b]{0.45\textwidth}
\begin{tikzpicture}[domain=0:2]
\begin{axis}[
axis lines=middle,
xlabel=$x$, xlabel style={at=(current axis.right of origin), anchor=west},
ylabel=$\tau^*(x)$, ylabel style={at=(current axis.above origin), anchor=south},
xtick={0,0.625},
xticklabels={0,$H$},
xmin=-0.5,
xmax=1.4,
ymajorticks=false,
%ytick={0},
%yticklabels={$0$},
%extra y ticks={0},
extra x ticks={0},
ymin=-0.02,
ymax=1
]
\draw[black!50!green,fill=black!50!green,fill opacity=0.2] (axis cs:0.625,0) circle (.65ex);
\addplot[thick,white!20!purple,domain=0.63:1.35]{0.9}; % node [pos=0.5,above]{$+\infty$};
\addplot[thick,white!20!purple,domain=-0.5:0.62]{0.9} node [pos=0.3,above]{$+\infty$};
\end{axis}
\end{tikzpicture}
\caption{Legendre transform of $\tau$}
\label{fig:subsec:31tau*}
\end{subfigure}
\caption{Scenario of Subsec.~\ref{subsec:31}}
\label{fig:subsec:31}
\end{figure}
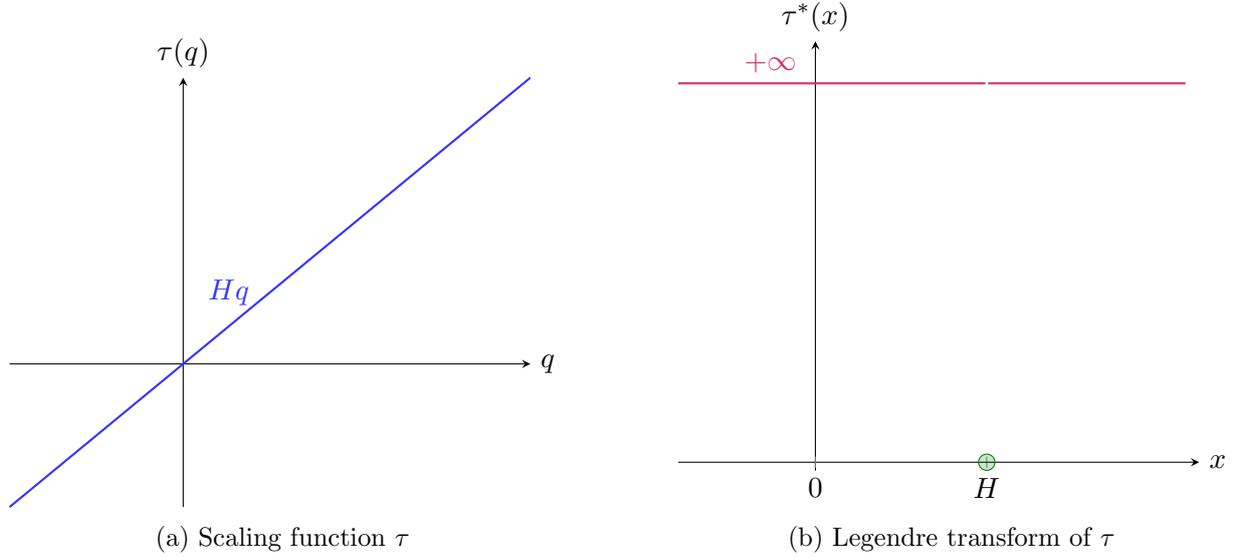

By Prop.~\ref{prop:logYlimit} we have that for any $\varepsilon>0$
\begin{equation}\label{e:H2}
P \left( t^{H-\varepsilon} < |X(t)| < t^{H+\varepsilon} \right) \to 1,
\end{equation}
as $t\to \infty$. If we apply Thm.~\ref{thm:mainGEapp} and put $A=(H-\varepsilon, H+\varepsilon)$ in \eqref{e:LDPsetA}, we get that for any $\varepsilon>0$
\begin{equation*}
\begin{aligned}
- \inf_{x \in \{H\} } \tau^* (x) &\leq \liminf_{t\to \infty} \frac{1}{\log t} \log P \left( H-\varepsilon < \frac{\log |X(t)|}{\log t} < H+\varepsilon \right)\\
&\leq \limsup_{t\to \infty} \frac{1}{\log t} \log P \left( H-\varepsilon < \frac{\log |X(t)|}{\log t} < H+\varepsilon \right) \leq - \inf_{x \in [H-\varepsilon, H+\varepsilon]} \tau^* (x).
\end{aligned}
\end{equation*}
By \eqref{e:ext}, the both extremes are equal to $\tau^*(0)=0$, and therefore, consistent with \eqref{e:H2},
\begin{equation*}%\label{e:H1}
\lim_{t\to \infty} \frac{1}{\log t} \log P \left( t^{H-\varepsilon} <  |X(t)| < t^{H+\varepsilon} \right) = 0.
\end{equation*}
For any other Borel set $A$ such that $H\notin A$ we have $\inf_{x \in \Int(A)\cap E} \tau^*(x)= \inf_{x \in \cl(A)} \tau^*(x) = \infty$, hence
\begin{equation*}
\lim_{t\to \infty} \frac{1}{\log t} \log P \left( \frac{\log |X(t)|}{\log t} \in A \right) = -\infty.
\end{equation*}
This implies that for any $m>0$, we can take $t$ large enough so that $P \left(\log |X(t)|/\log t \in A \right) \leq t^{-m}$. Because we are assuming $H\notin A$, the probability of observing any scale beside $t^H$ decays to zero faster than any negative power of $t$.

This scenario illustrates the fact that in the usual limit theorems when all the moments converge, there is a single rate of growth, or a single scale. Note, however, that the assumption \eqref{e:mooconvsubsec41} and the assumption that all the moments are finite is quite restrictive.

\subsection{Limit theorems with possibly infinite moments where finite moments converge}\label{subsec:32}

In the previous scenario we assumed that all the moments exist, both of positive and negative order. Note that for a random variable $Z$ with an absolutely continuous distribution we have $\E|Z|^{-1}=\infty$ as soon as the density is continuous and positive at zero (see e.g.~\cite{piegorsch1985existence}). This implies $\E|Z|^{q}=\infty$ for any $q\leq -1$. Hence, for usual distributions absolute moments of a negative order less than $-1$ are infinite. We now extend the setting of the previous scenario to include possible infinite moments, both of positive and negative order.

Suppose that for the process $X$ we have $\overline{q}=\overline{q}(X)>0$ and $\underline{q}=\underline{q}(X)<0$, with $\overline{q}(X)$ and $\underline{q}(X)$ defined in \eqref{e:qunderover}. Assume further that $X$ satisfies the limit theorem as in \eqref{limittheorem}, and that finite moments converge, i.e.~\eqref{e:mooconvsubsec41} holds for every $q\in (\underline{q}, \overline{q})$ and $t\geq 0$. Note that this scenario applies to any self-similar process $X$ as \eqref{limittheorem} and \eqref{e:mooconvsubsec41} then trivially hold with $Z=X$. The scaling function \eqref{deftau} of $X$ is \cite[Thm.~1]{GLST2019Bernoulli}
\begin{equation*}
    \tau(q)= H q, \quad q \in (\underline{q}, \overline{q}),
\end{equation*}
and $\tau(q)=\infty$ for $q<\underline{q}$ and $q>\overline{q}$. The Legendre transform \eqref{e:LFtransform} of $\tau$ may be computed as follows (see Fig.~\ref{fig:subsec:32}):
\begin{equation*}
\tau^*(x) = \max \left\{ \sup_{q \in (\underline{q}, \overline{q})} \left\{ q x - Hq \right\}, \ \sup_{q \notin (\underline{q}, \overline{q})} \left\{ q x - \tau(q) \right\} \right\}= \begin{cases}
\underline{q}(x-H), & \text{ if } x<H,\\
0, & \text{ if } x=H,\\
\overline{q}(x-H), & \text{ if } x>H.\\
\end{cases}
\end{equation*}
This implicitly includes the case when $\underline{q}=-\infty$ or $\overline{q}=\infty$ or both. The set of exposed points of $\tau^*$ is $E=\{H\}$.

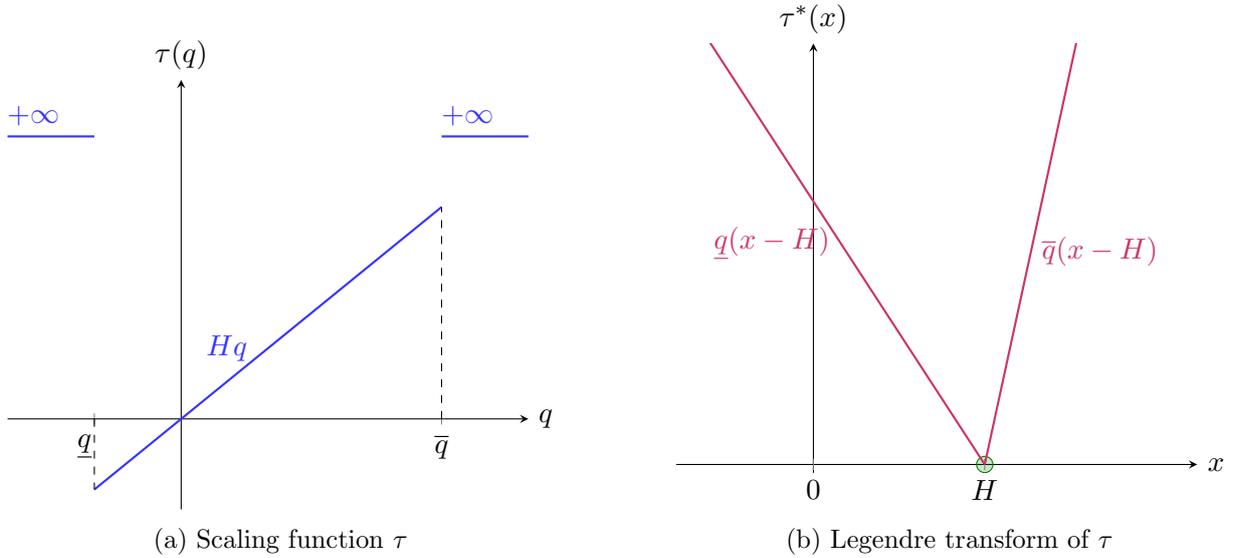
\begin{figure}[h!]
\centering
\begin{subfigure}[b]{0.45\textwidth}
\begin{tikzpicture}[domain=-2:2]
\begin{axis}[
axis lines=middle,
xlabel=$q$, xlabel style={at=(current axis.right of origin), anchor=west},
ylabel=$\tau(q)$, ylabel style={at=(current axis.above origin), anchor=south},
xtick={-1,0,3},
xticklabels={$\underline{q}\ \ $, $0$, $\overline{q}$},
xmin=-2,
xmax=4,
ymajorticks=false,
ymin=-0.8,
ymax=3
]
\addplot[thick,white!20!blue,domain=-1:3]{0.625*x} node [pos=0.5,left]{$Hq\ $};
\addplot[dashed] coordinates {(-1,0) (-1,-0.625)};
\addplot[dashed] coordinates {(3,0) (3,1.875)};
\addplot[thick,white!20!blue,domain=-2:-1]{2.5} node [pos=0.3,above]{$+\infty$};
\addplot[thick,white!20!blue,domain=3:4]{2.5} node [pos=0.3,above]{$+\infty$};
\end{axis}
\end{tikzpicture}
\caption{Scaling function $\tau$}
\label{fig:subsec:32tau}
\end{subfigure}
\hfill
\begin{subfigure}[b]{0.45\textwidth}
\begin{tikzpicture}[domain=0:2]
\begin{axis}[
axis lines=middle,
xlabel=$x$, xlabel style={at=(current axis.right of origin), anchor=west},
ylabel=$\tau^*(x)$, ylabel style={at=(current axis.above origin), anchor=south},
xtick={0,0.625},
xticklabels={0,$H$},
xmin=-0.5,
xmax=1.4,
ymajorticks=false,
%ytick={0},
%yticklabels={$0$},
%extra y ticks={0},
extra x ticks={0},
ymin=-0.02,
ymax=1
]
\draw[black!50!green,fill=black!50!green,fill opacity=0.2] (axis cs:0.625,0) circle (.65ex);
\addplot[thick,white!20!purple,domain=-2:0.625]{-1*(x-0.625)} node [pos=0.8,left]{$\underline{q}(x-H)$};
\addplot[thick,white!20!purple,domain=0.625:4]{3*(x-0.625)} node [pos=0.05,right]{$\overline{q}(x-H)$};
\end{axis}
\end{tikzpicture}
\caption{Legendre transform of $\tau$}
\label{fig:subsec:32tau*}
\end{subfigure}
\caption{Scenario of Subsec.~\ref{subsec:32}}
\label{fig:subsec:32}
\end{figure}

If we take $A=(H+\varepsilon, \infty)$ and $A=(-\infty, H-\varepsilon)$, we respectively obtain the bounds
\begin{equation*}
\begin{aligned}
-\infty &\leq \liminf_{t\to \infty} \frac{1}{\log t} \log P \left( \frac{\log |X(t)|}{\log t} > H+\varepsilon \right)\leq \limsup_{t\to \infty} \frac{1}{\log t} \log P \left( \frac{\log |X(t)|}{\log t} > H+\varepsilon \right) \leq - \overline{q} \varepsilon,\\
-\infty &\leq \liminf_{t\to \infty} \frac{1}{\log t} \log P \left( \frac{\log |X(t)|}{\log t} < H-\varepsilon \right)\leq \limsup_{t\to \infty} \frac{1}{\log t} \log P \left( \frac{\log |X(t)|}{\log t} < H-\varepsilon \right) \leq \underline{q} \varepsilon,
\end{aligned}
\end{equation*}
which shows that for any $\delta>0$ we can take $t$ large enough so that
\begin{align}
P \left( |X(t)| > t^{H+\varepsilon} \right) \leq t^{- \overline{q} \varepsilon + \delta},\nonumber\\
P \left( |X(t)| < t^{H-\varepsilon} \right) \leq t^{\underline{q} \varepsilon+ \delta}.\label{e:sc2:bound2}
\end{align}
%We will write this shortly as
%\begin{equation*}
%\begin{aligned}
%P \left( |Y(n)| > n^{H+\varepsilon} \right) \lesssim n^{- \overline{q} \varepsilon},\\
%P \left( |Y(n)| < n^{H-\varepsilon} \right) \lesssim n^{\underline{q} \varepsilon}.
%\end{aligned}
%\end{equation*}
If $\overline{q}=\infty$, the first bound shows that $X$ does not exhibit rates greater than $H$ when power-law decaying probabilities are considered. If $\underline{q}=-\infty$, the second bound shows that $X$ does not exhibit rates less than $H$ and we have the situation considered in Subsec.~\ref{subsec:31}.

Consider for example the case where $\{X(t)\}$ is fractional Brownian motion. In this case, $\underline{q}=-1$ and $\overline{q}=\infty$. Since $\overline{q}=\infty$, we get that $P \left( |X(t)| > t^{H+\varepsilon} \right) \to 0$, faster than any negative power of $t$. But since $\underline{q}=-1$, for any $\delta>0$ we have $P \left( |X(t)| < t^{H-\varepsilon} \right)\leq t^{-\varepsilon+\delta}$ from \eqref{e:sc2:bound2} and we cannot conclude that the rates less than $H-\varepsilon$ are negligible on the power probability scale.

These bounds may not be sharp in general. However, as soon as positive order moments are finite and converge, we can conclude that $X(t)$ cannot grow faster than $t^{H+\varepsilon}$ with a probability decaying as some power of $t$.

\subsection{A simple biscale example}\label{subsec:33}

In this subsection we construct an example of a sequence that has two rates of growth. This closely resembles the behavior of the integrated supOU processes which will be considered in the next section. Suppose that $X(t)$, $t \in \N$, is an independent sequence given by
\begin{equation*}
X(t) = \begin{cases}
t^{H}, & \text{ with probability } 1-t^{-a},\\
t^b, & \text{ with probability } t^{-a},
\end{cases}
\end{equation*}
where $0<H<b$ and $a>0$. The scaling function for $q\in \R$ is given by
\begin{equation}\label{e:biscaletau}
\tau(q) = \lim_{t\to \infty} \frac{1}{\log t} \log \left( t^{H q} \left(1-t^{-a} \right) + t^{bq}t^{-a} \right)=\begin{cases}
H q, & \text{ if } q\leq \frac{a}{b-H},\\
bq-a, & \text{ if } q> \frac{a}{b-H},
\end{cases}
\end{equation}
and we have intermittency (see \eqref{e:inter}) since $q \mapsto \tau(q)/q$ is strictly increasing on $\left(a/(b-H), \infty\right)$. One can compute that
\begin{align*}
&\tau^*(x) = \max \left\{ \sup_{q \leq a/(b-H)} \left\{ q \left(x - H\right) \right\}, \ \sup_{q >a/(b-H)} \left\{ q \left(x - b\right) + a \right\} \right\}\\
&=\begin{cases}
\max \left\{ \infty, \ \frac{a}{b-H} x - \frac{aH}{b-H} \right\}, & \text{ if } x < H,\\
\max \left\{ \frac{a}{b-H} x - \frac{aH}{b-H}, \ \frac{a}{b-H} x - \frac{aH}{b-H} \right\}, & \text{ if } H \leq x \leq b,\\
\max \left\{ \frac{a}{b-H} x - \frac{aH}{b-H}, \ \infty \right\}, & \text{ if } x > b,
\end{cases}=\begin{cases}
\frac{a}{b-H} x - \frac{aH}{b-H}, & \text{ if } H \leq x \leq b,\\
\infty, & \text{ otherwise},
\end{cases}
\end{align*}
and the set of exposed points is $E=\{ H, b \}$ (see Fig.~\ref{fig:subsec:33}).

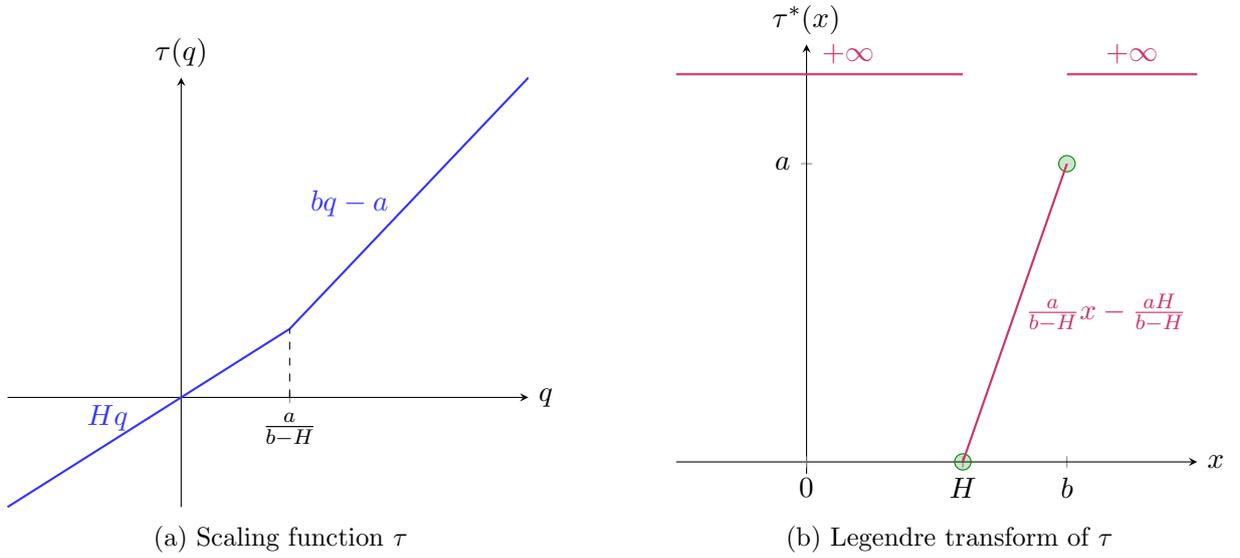
\begin{figure}[h!]
\centering
\begin{subfigure}[b]{0.45\textwidth}
\begin{tikzpicture}[domain=-2:2]
\begin{axis}[
axis lines=middle,
xlabel=$q$, xlabel style={at=(current axis.right of origin), anchor=west},
ylabel=$\tau(q)$, ylabel style={at=(current axis.above origin), anchor=south},
xtick={0,1.25},
xticklabels={$0$,$\frac{a}{b-H}$},
xmin=-2,
xmax=4,
ymajorticks=false
]
\addplot[thick,white!20!blue,domain=-2:1.25]{0.6*x} node [pos=0.5,left]{$Hq\ $};
\addplot[thick,white!20!blue,domain=1.25:4]{x-0.5} node [pos=0.5,left]{$bq-a\ $};
\addplot[dashed] coordinates {(1.25,0) (1.25,0.75)};
\end{axis}
\end{tikzpicture}
\caption{Scaling function $\tau$}
\label{fig:subsec:33tau}
\end{subfigure}
\hfill
\begin{subfigure}[b]{0.45\textwidth}
\begin{tikzpicture}[domain=0:2]
\begin{axis}[
axis lines=middle,
xlabel=$x$, xlabel style={at=(current axis.right of origin), anchor=west},
ylabel=$\tau^*(x)$, ylabel style={at=(current axis.above origin), anchor=south},
xtick={0.6,1},
xticklabels={$H$,$b$},
xmin=-0.5,
xmax=1.5,
ytick={0,0.5},
yticklabels={$0$,$a$},
%extra y ticks={0},
extra x ticks={0},
ymin=-0.02,
ymax=0.7
]
\draw[black!50!green,fill=black!50!green,fill opacity=0.2] (axis cs:0.6,0) circle (.65ex);
\draw[black!50!green,fill=black!50!green,fill opacity=0.2] (axis cs:1,0.5) circle (.65ex);
\addplot[thick,white!20!purple,domain=0.6:1]{1.25*x-0.75} node [pos=0.5,right]{$\frac{a}{b-H} x - \frac{aH}{b-H}$};
\addplot[thick,white!20!purple,domain=-0.5:0.6]{0.65} node [pos=0.6,above]{$+\infty$};
\addplot[thick,white!20!purple,domain=1:1.5]{0.65} node [pos=0.5,above]{$+\infty$};
\end{axis}
\end{tikzpicture}
\caption{Legendre transform of $\tau$}
\label{fig:subsec:33tau*}
\end{subfigure}
\caption{Example of Subsec.~\ref{subsec:33}}
\label{fig:subsec:33}
\end{figure}

If we take $A=\left(b- \varepsilon, b+ \varepsilon\right)$, then
\begin{equation*}
\begin{aligned}
- a = -\tau^*(b) &\leq \liminf_{t\to \infty} \frac{1}{\log t} \log P \left( t^{b-\varepsilon} <  |X(t)| < t^{b+\varepsilon} \right)\\
&\leq \limsup_{t\to \infty} \frac{1}{\log t} \log P \left( t^{b-\varepsilon} <  |X(t)| < t^{b+\varepsilon} \right) \leq -\tau^*(b-\varepsilon) = - \left(a - \frac{\varepsilon a}{b-H} \right).
\end{aligned}
\end{equation*}
This way we can conclude from the behavior of moments and the scaling function that two typical rates appear for $X$: one of the order $H$ and one of the order $b$. The first one is dominant, but the second is also relevant as it is not negligible when power law decaying probabilities are considered. On the other hand, if $A=(b+\varepsilon, \infty)$, then
\begin{equation*}
- \infty \leq \liminf_{t\to \infty} \frac{1}{\log t} \log P \left( |X(t)| > t^{b+\varepsilon} \right) \leq \limsup_{t\to \infty} \frac{1}{\log t} \log P \left( |X(t)| > t^{b+\varepsilon} \right) \leq - \infty,
\end{equation*}
showing that $P \left( |X(t)| > t^{b+\varepsilon} \right) \to 0$  faster than any negative power of $t$. For $A=(-\infty, H - \varepsilon)$
\begin{equation*}
- \infty \leq \liminf_{t\to \infty} \frac{1}{\log t} \log P \left( |X(t)| < t^{H-\varepsilon} \right) \leq \limsup_{t\to \infty} \frac{1}{\log t} \log P \left( |X(t)| < t^{H-\varepsilon} \right) \leq - \infty,
\end{equation*}
showing that $P \left( |X(t)| < t^{H-\varepsilon} \right) \to 0$ faster than any negative power of $t$. Hence, there are no rates greater than $b+\varepsilon$ and less than $H-\varepsilon$ for any $\varepsilon>0$, consistent with the definition of $X(t)$.

Let now $c$ and $\varepsilon>0$ be such that $H < c-\varepsilon < c+\varepsilon < b$. For $A=(c-\varepsilon, c+\varepsilon)$ we have
\begin{equation}\label{e:nomiddlerates}
\begin{aligned}
- \infty &\leq \liminf_{t\to \infty} \frac{1}{\log t} \log P \left( t^{c-\varepsilon} <  |X(t)| < t^{c+\varepsilon} \right)\\
&\leq \limsup_{t\to \infty} \frac{1}{\log t} \log P \left( t^{c-\varepsilon} <  |X(t)| < t^{c+\varepsilon} \right) \leq - \left( \frac{a}{b-H} (c-\varepsilon) - \frac{aH}{b-H} \right).
\end{aligned}
\end{equation}
The right-hand side is negative so that $P \left( t^{c-\varepsilon} <  |X(t)| < t^{c+\varepsilon} \right)\to 0$ as $t\to \infty$. By the definition of $X$ we know it does not exhibit scales $t^c$ for $H < c < b$ and $P \left( t^{c-\varepsilon} <  |X(t)| < t^{c+\varepsilon} \right)= 0$. One would expect to get $-\infty$ on the right-hand side of \eqref{e:nomiddlerates}, hence the bound obtained is not the best possible (see also the discussion in Sec.~\ref{sec6}).

\section{SupOU processes}\label{subsec:34}

The \textit{supOU} processes have been introduced in \cite{bn2001} as a strictly stationary process $Y=\{Y(t), \, t\in \R\}$ represented by the stochastic integral
\begin{equation*}%\label{supOU}
Y(t)= \int_{\R_+} \int_{\R} e^{-\xi t + s} \mathbf{1}_{[0,\infty)}(\xi t -s) \Lambda(d\xi,ds).
\end{equation*}
Here, $\Lambda$ is a homogeneous infinitely divisible random measure (\textit{L\'evy basis}) on $\R_+ \times \R$ such that $\log \E e^{i \zeta \Lambda(A) } =  \left( \pi \times Leb \right) (A) \kappa_{L}(\zeta)$, for $A \in \mathcal{B} \left(\R_+ \times \R\right)$, where $\pi$ is a probability measure on $\R_+$, $Leb$ denotes Lebesgue measure and $\kappa_{L}$ is the cumulant function $\kappa_{L} (\zeta)= \log \E e^{i \zeta L(1) }$ of some infinitely divisible random variable $L(1)$ with L\'evy-Khintchine triplet $(a,b,\mu)$, i.e.
\begin{equation*}%\label{kappacumfun}
\kappa_{L}(\zeta) = i\zeta a -\frac{\zeta ^{2}}{2} b  +\int_{\R}\left( e^{i\zeta x}-1-i\zeta x \mathbf{1}_{[-1,1]}(x)\right) \mu(dx).
\end{equation*}

To explain the definition, recall that the L\'evy driven \textit{Ornstein-Uhlenbeck type (OU) process} is a process $\{V(t), \, t\in \R\}$ defined by
\begin{equation*}%\label{OUonRdef1}
  V(t) = V^{(\lambda, L)}(t) = \int_{\R} e^{-\lambda t - s} \mathbf{1}_{[0,\infty)}(\lambda t-s) dL (s).
\end{equation*}
where $L$ is a two-sided L\'evy process satisfying $\E \log \left(1+ \left| L(1) \right| \right)< \infty$ and $\lambda>0$. The correlation function of the OU type process is always exponential. One may obtain a bit more flexible correlation structure by considering superposition $\sum_{k=1}^m w_k V^{(\lambda_k, L)}(t)$ of independent OU type processes for some weights $w_k$, $k=1,\dots,n$. The supOU processes generalize this idea to infinite superpositions. One can also view superposition as averaging over $\lambda$ randomized according to $\pi$ which formally corresponds to
\begin{equation*}
 Y(t) = \int_{\R_+} \int_{\R} e^{-\xi t+s} \mathbf{1}_{[0,\infty)} (\xi t -s) dL(s) \pi(d\xi).
\end{equation*}
In the \textit{characteristic quadruple}
\begin{equation}\label{quadruple}
(a,b,\mu,\pi),
\end{equation}
$(a,b,\mu)$ determine the marginal distribution of $X$, while the dependence structure is controlled by $\pi$. For example, by taking $\pi$ with density $p$ satisfying
\begin{equation}\label{regvarofp}
p (x) \sim \alpha \ell(x^{-1}) x^{\alpha-1}, \quad \text{ as } x \to 0.
\end{equation}
for some $\alpha>0$ and some slowly varying function $\ell$, we get that the correlation function satisfies
\begin{equation*}%\label{regvarofACR}
r( \tau) \sim \Gamma(1+\alpha) \ell(\tau) \tau^{-\alpha}, \qquad \text{ as } \tau \to \infty.
\end{equation*}
In particular, for $\alpha \in (0,1)$ one obtains \textit{long-range dependence}, that is, a nonintegrable correlation function. More details about supOU processes can be found in \cite{bn2001,barndorff2011multivariate,barndorff2018ambit,GLST2019Bernoulli}.

\subsection{Limit theorems and intermittency}
Since supOU processes are continuous time processes, instead of partial sums, one may consider the \textit{integrated supOU process} $X=\{X(t), \, t \geq 0\}$ defined by
\begin{equation}\label{integratedsupOU}
X(t) = \int_0^t Y(s) ds,
\end{equation}
which has stationary increments. The limit theorems have been established in \cite{GLT2019Limit} for the finite variance integrated process and in \cite{GLT2019LimitInfVar} for the infinite variance case. Somewhat surprisingly, the type of the limiting process may depend on the behavior of the L\'evy measure $\mu$ near the origin. When this happens, we quantify this behavior by assuming that
\begin{equation}\label{LevyMCond}
\mu \left( [x, \infty) \right) \sim c^+ x^{-\beta} \ \text{ and } \ \mu \left( (-\infty, -x] \right) \sim c^- x^{-\beta} \ \text{  as } x \to 0,
\end{equation}
for some $\beta>0$, $c^+, c^- \geq 0$, $c^++c^->0$. In particular, if \eqref{LevyMCond} holds, then $\beta$ is the Blumenthal-Getoor index of $\mu$: $\beta_{BG} = \inf \left\{\gamma \geq 0 : \int_{|x|\leq 1} |x|^\gamma \mu(dx) < \infty \right\}$.

\begin{theorem}[Thms.~3.1-3.4 in \cite{GLT2019Limit}]\label{thm:supOUlimittheoremfinvar}
Suppose that $Y$ is a supOU process with zero mean, finite variance and the characteristic quadruple \eqref{quadruple} such that $\pi$ has a density $p$ satisfying \eqref{regvarofp} for some $\alpha>0$ and some slowly varying function $\ell$.
Then for some slowly varying function $\widehat{\ell}$, the integrated process \eqref{integratedsupOU} satisfies
\begin{equation}\label{int:limittheorem}
\left\{ \frac{1}{T^H \widehat{\ell}(T)} X(Tt) \right \} \overset{fdd}{\to} \left\{ Z(t) \right\},
\end{equation}
if one of the following holds:
\begin{enumerate}[(i)]
\item $b>0$ and $\alpha \in (0,1)$, in which case $H=1-\alpha/2$ and $Z$ is a fractional Brownian motion,
\item $b=0$, $\alpha \in (0,1)$ and $\beta_{BG}<1+\alpha$, in which case $H=1/(1+\alpha)$ and $Z$ is a stable L\'evy process,
\item $b=0$, $\alpha \in (0,1)$ and \eqref{LevyMCond} holds with $1+\alpha<\beta<2$, in which case $H=1-\alpha/\beta$ and $Z$ is a $\beta$-stable process with dependent increments,
\item $\alpha >1$, in which case $H=1/2$ and $Z$ is a Brownian motion.
\end{enumerate}
\end{theorem}

The $\beta$-stable process from \textit{(iii)} in Theorem \ref{thm:supOUlimittheoremfinvar} has a stochastic integral representation
\begin{equation*}%\label{Zalphabeta}
Z(t) = \int_{\R_+} \int_{\R} \left( \mathfrak{f}(\xi, t-s) - \mathfrak{f}(\xi, -s) \right) K(d\xi, ds),
\end{equation*}
where $\mathfrak{f}$ is given by
\begin{equation*}%\label{mff}
\mff(x,u) = \begin{cases}
1-e^{-xu}, & \text{ if } x>0 \text{ and } u>0,\\
0, & \text{ otherwise},
\end{cases}
\end{equation*}
and $K$ is a $\beta$-stable L\'evy basis on $\R_+ \times \R$ with control measure $k(d\xi, ds)=\alpha \xi^{\alpha} d\xi ds$. It is a $H=1-\alpha/\beta$ self-similar process with stationary increments.

The convergence of finite dimensional distributions in \eqref{int:limittheorem} can be extended to weak convergence in some cases \cite[Thm.~3.5]{GLT2019Limit}. The limit theorems in the infinite variance case were obtained in \cite{GLT2019LimitInfVar} and cover the case when the marginal distribution of the supOU process $Y$ belongs to the domain of attraction of a stable law, that is, $Y(1)$ has balanced regularly varying tails:
\begin{equation}\label{regvarofX}
P(Y(1)>x) \sim p k(x) x^{-\gamma} \quad \text{and} \quad P(Y(1)\leq - x) \sim q k(x) x^{-\gamma}, \quad  \text{ as } x\to \infty,
\end{equation}
for some $p,q \geq 0$, $p+q>0$, $0<\gamma<2$ and some slowly varying function $k$. If $\gamma=1$, assume that $p=q$. The limiting behavior depends additionally on the regular variation index $\gamma$ of the marginal distribution of the supOU process. The class of possible limiting processes is the same as in Thm.~\ref{thm:supOUlimittheoremfinvar}, except for Brownian motion which never appears in the limit in the infinite variance case. See \cite{GLT2019LimitInfVar} for details.

That integrated supOU processes may be intermittent has been showed in \cite{GLST2016JSP,GLST2019Bernoulli}. Note, however, that if the supOU is Gaussian, then there is no intermittency.

\begin{theorem}\label{thm:supOUlimittheoremmoments}
Suppose in addition to the assumptions of Thm.~\ref{thm:supOUlimittheoremfinvar} that there exists $a>0$ such that $\E e^{a |Y(t)|} <\infty$ and that $\alpha$ is integer if $\alpha>1$ in \eqref{regvarofp}. If $Y$ is not purely Gaussian, then the scaling function $\tau$ of the integrated process $X$ is
\begin{equation}\label{tausupOU}
\tau(q) = \begin{cases}
H q, & 0\leq q\leq \frac{\alpha}{1-H},\\
q-\alpha, & q\geq \frac{\alpha}{1-H},
\end{cases}
\end{equation}
where $H$ is as in Theorem \ref{thm:supOUlimittheoremfinvar}. If $Y$ is purely Gaussian, then $\tau(q) = \left(1-\min\{1,\alpha\}/2\right) q$ for every $q\geq 0$.
\end{theorem}

Eq.~\eqref{tausupOU} implies that $q \mapsto \tau(q)/q$ is strictly increasing on $(\alpha/(1-H), \infty)$ and hence the integrated supOU is intermittent when \eqref{tausupOU} holds. Note that this corresponds to \eqref{e:biscaletau} with $a=\alpha$ and $b=1$. If the supOU process is purely Gaussian, then there is no intermittency. In the infinite variance case, the range of finite positive order moments is limited and intermittency appears only in specific scenarios (see \cite{GLT2019MomInfVar} for details).

\subsection{Large deviations of the rate of growth}

From Thm.~\ref{thm:mainGEapp} we can prove the following large deviation bounds for the rate of growth.

\begin{theorem}\label{thm:supOUfv}
Suppose that the assumptions of Thm.~\ref{thm:supOUlimittheoremmoments} hold for the non-Gaussian supOU process and that $0 \in \Int (\mathcal{D}_\tau)$ where $\mathcal{D}_\tau = \left\{ q \in \R : \tau(q)<\infty \right\}$ and $\tau$ is the scaling function of the integrated process $X$. Then for a Borel set $A\subset \R$
\begin{equation}\label{e:LDPsetA}
\begin{aligned}
- \inf_{x \in \Int(A) \cap \{H,1\} } \tau^* (x) &\leq \liminf_{t\to \infty} \frac{1}{\log t} \log P \left( \frac{\log |X(t)|}{\log t} \in A \right)\\
&\leq \limsup_{t\to \infty} \frac{1}{\log t}  \log P \left( \frac{\log |X(t)|}{\log t} \in A \right) \leq - \inf_{x \in \cl (A)} \tau^* (x),
\end{aligned}
\end{equation}
where $H$ is as in Thm.~\ref{thm:supOUlimittheoremfinvar} and
\begin{equation}\label{e:tau*intheproof}
\tau^*(x) = \begin{cases}
\max \left\{\sup_{q <0} \left\{ qx - \tau(q) \right\}, \ 0 \right\}, & \text{ if } x < H,\\
\frac{\alpha}{1-H} x - \frac{\alpha H}{1-H}, & \text{ if } H \leq x \leq 1,\\
\infty, & \text{ if } x > 1.
\end{cases}
\end{equation}
\end{theorem}

\begin{proof}
To apply Thm.~\ref{thm:mainGEapp}, we first compute the Legendre transform $\tau^*$ from the expression for $\tau$ given in Thm.~\ref{thm:supOUlimittheoremmoments}:
\begin{align}
&\tau^*(x) = \max \left\{ \sup_{q <0} \left\{ qx - \tau(q) \right\}, \ \sup_{0\leq q \leq \alpha/(1-H)} \left\{ q \left(x - H\right) \right\}, \ \sup_{q >\alpha/(1-H)} \left\{ q \left(x - 1\right) + \alpha \right\} \right\}\nonumber\\
&\quad =\begin{cases}
\max \left\{\sup_{q <0} \left\{ qx - \tau(q) \right\}, \ 0, \ \frac{\alpha}{1-H} x - \frac{\alpha H}{1-H} \right\}, & \text{ if } x < H,\\
\max \left\{\sup_{q <0} \left\{ qx - \tau(q) \right\}, \ \frac{\alpha}{1-H} x - \frac{\alpha H}{1-H}, \ \frac{\alpha}{1-H} x - \frac{\alpha H}{1-H} \right\}, & \text{ if } H \leq x \leq 1,\\
\max \left\{\sup_{q <0} \left\{ qx - \tau(q) \right\}, \ \frac{\alpha}{1-H} x - \frac{\alpha H}{1-H}, \ \infty \right\}, & \text{ if } x > 1.
\end{cases}\label{e:m3}
\end{align}
Computing $\tau^*$ requires knowing $\tau(q)$ for negative $q$ but we avoid this by using the bound given in Prop.~\ref{propertiesoftau}\textit{(iv)}. Since $\inf_{q'>0} \frac{\tau(q')}{q'} = \min \{ \inf_{0< q'\leq \alpha/(1-H)} H, \ \allowbreak \inf_{q'\geq \alpha/(1-H)} \left(1- \frac{\alpha}{q'}\right) \} = H$, we get from \eqref{e:sfnegqbound} that for $q<0$, $\tau(q) \geq H q$. By using this bound we get
\begin{equation*}
\sup_{q <0} \left\{ qx - \tau(q) \right\} \leq \sup_{q <0} \left\{ qx - H q \right\} =\begin{cases}
\infty, & \text{ if } x < H,\\
0, & \text{ if } x \geq H.
\end{cases}
\end{equation*}
The bound $\infty$ for $x < H$ is not useful, but the second bound $0$ is useful for \eqref{e:m3} and yields \eqref{e:tau*intheproof}. By Thm.~\ref{thm:mainGEapp}, \eqref{e:LDPsetA} follows with the infimum on the left-hand side taken over $ \Int(A) \cap E$ for $E$ the set of exposed points of $\tau^*$ whose exposing hyperplane belongs to $\Int (\mathcal{D}_{\tau})$. In our case, however, $\{H, 1\} \subset E$ giving \eqref{e:LDPsetA}.
\end{proof}

We note here two special cases. For $A=\left(1- \varepsilon, 1+ \varepsilon\right)$ we get from \eqref{e:LDPsetA}
\begin{equation*}
\begin{aligned}
-\alpha =- \tau^* \left( 1 \right) &\leq\liminf_{t\to \infty} \frac{P \left( t^{1-\varepsilon} <  |X(t)| < t^{1+\varepsilon} \right)}{\log t}\\
&\leq \limsup_{t\to \infty} \frac{\log P \left( t^{1-\varepsilon} <  |X(t)| < t^{1+\varepsilon} \right) }{\log t} \leq - \tau^* \left( 1-\varepsilon \right) = -\alpha + \frac{\varepsilon \alpha}{1-H},
\end{aligned}
\end{equation*}
and for $A=(1+\varepsilon, \infty)$ we obtain
\begin{equation}\label{e:thm51ii}
\lim_{t\to \infty} \frac{\log P \left( |X(t)| > t^{1+\varepsilon} \right)}{\log t}  = -\infty,
\end{equation}
which shows that the probability of rates greater than $1$ decays faster than any power of $t$. Recall that Prop.~\ref{prop:logYlimit} implies that
\begin{equation*}
\lim_{t\to \infty} P  \left( t^{H-\varepsilon} <  |X(t)| < t^{H+\varepsilon} \right) = 1.
\end{equation*}
On the other hand, Thm.~\ref{thm:supOUfv} shows that while the integrated supOU process has one typical order of magnitude, it may also exhibit values of the greater order on a random island probability of which decays as a power function of $t$. This behavior of the process is responsible for the unusual behavior of the moments and causes a change-point in the scaling function.

Fig.~\ref{fig:subsec:fvsup} show the scaling function given \eqref{tausupOU} and its Legendre transform computed in \eqref{e:tau*intheproof}. On the abscissa one reads the rates of growth: $H$ and $1$ for the case of Thm.~\ref{thm:supOUfv}. On the ordinate, one can read the probability of these rates: probability of rate $H$ is roughly $t^0=1$ and probability of rate $1$ is roughly $t^{-\alpha}$.

\begin{figure}[h!]
\centering
\begin{subfigure}[b]{0.45\textwidth}
\begin{tikzpicture}[domain=0:3]
\begin{axis}[
axis y line=middle, axis x line=bottom,
xlabel=$q$, xlabel style={at=(current axis.right of origin), anchor=west},
ylabel=$\tau(q)$, ylabel style={at=(current axis.above origin), anchor=south},
xtick={0,1.25},
xticklabels={$0$,$\frac{\alpha}{1-H}$},
xmin=-0,
xmax=4,
ymajorticks=false
]
\addplot[thick,white!20!blue,domain=0:1.25]{0.6*x} node [pos=0.5,left]{$Hq\ $};
\addplot[thick,white!20!blue,domain=1.25:3]{x-0.5} node [pos=0.5,left]{$bq-a\ $};
\addplot[dashed] coordinates {(1.25,0) (1.25,0.75)};
\end{axis}
\end{tikzpicture}
\caption{Scaling function $\tau$}
\label{fig:subsec:fvsuptau}
\end{subfigure}
\hfill
\begin{subfigure}[b]{0.45\textwidth}
\begin{tikzpicture}[domain=0:2]
\begin{axis}[
axis lines=middle,
xlabel=$x$, xlabel style={at=(current axis.right of origin), anchor=west},
ylabel=$\tau^*(x)$, ylabel style={at=(current axis.above origin), anchor=south},
xtick={0.6,1},
xticklabels={$H$,$1$},
xmin=-0.02,
xmax=1.5,
ytick={0,0.5},
yticklabels={$0$,$\alpha$},
extra y ticks={0},
ymin=-0.02,
ymax=0.7
]
\draw[black!50!green,fill=black!50!green,fill opacity=0.2] (axis cs:0.6,0) circle (.65ex);
\draw[black!50!green,fill=black!50!green,fill opacity=0.2] (axis cs:1,0.5) circle (.65ex);
\addplot[thick,white!20!purple,domain=0.6:1]{1.25*x-0.75} node [pos=0.5,right]{$\frac{\alpha}{1-H} x - \frac{\alpha H}{1-H}$};
%\addplot[thick,white!20!purple,domain=-0.5:0.6]{0.65} node [pos=0.6,above]{$+\infty$};
\addplot[thick,white!20!purple,domain=1:1.5]{0.65} node [pos=0.5,above]{$+\infty$};
\end{axis}
\end{tikzpicture}
\caption{Legendre transform of $\tau$}
\label{fig:subsec:fvsupLT}
\end{subfigure}
\caption{Finite variance supOU (Thm.~\ref{thm:supOUfv})}
\label{fig:subsec:fvsup}
\end{figure}
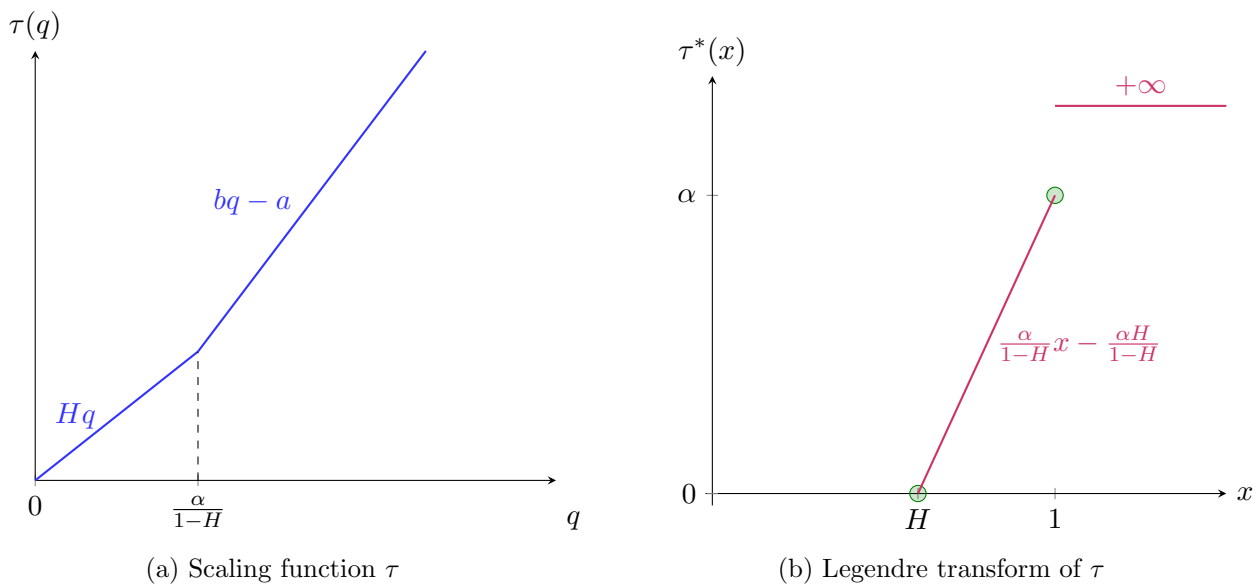

Note that the results of Thm.~\ref{thm:supOUfv} resemble those obtained for the example in Subsec.~\ref{subsec:33} with $a=\alpha$ and $b=1$. Here, however, we do not know whether
\begin{equation*}
\frac{1}{\log t} \log P \left( |X(t)| < t^{H-\varepsilon} \right) \to 0, \ \text{ as } t\to \infty,
\end{equation*}
since we do not know $\inf_{x \in A} \tau^*(x)$ for $A=(-\infty, H - \varepsilon)$ which would involve computing negative order moments of the integrated process.

\subsection{Infinite variance supOU process}
If the supOU process has infinite variance, the range of finite moments is limited but we can still prove the following.

\begin{theorem}\label{thm:supOUinfv}
Suppose $Y$ is a supOU process with the characteristic quadruple \eqref{quadruple} such that \eqref{regvarofX} holds with $0<\gamma<2$, $\E Y(1)=0$ if mean is finite, $\pi$ has a density $p$ satisfying \eqref{regvarofp} for some $\alpha>0$ and some slowly varying function $\ell$ and \eqref{LevyMCond} holds with $0\leq \beta<2$. Suppose also that $0 \in \Int (\mathcal{D}_\tau)$.% where $\mathcal{D}_\tau = \left\{ q \in \R : \tau(q)<\infty \right\}$ and $\tau$ is the scaling function of the integrated process $X^*$.
Then the following holds:
\begin{enumerate}[(i)]
\item If $b=0$ and $\beta<1+\alpha<\gamma$, then for any $\varepsilon>0$, $\varepsilon<1-1/(1+\alpha)$
\begin{align*}
&- \alpha \leq \liminf_{t\to \infty} \frac{1}{\log t} \log P \left( t^{1-\varepsilon} <  |X(t)| < t^{1+\varepsilon} \right)\\
&\hspace{3cm} \leq \limsup_{t\to \infty} \frac{1}{\log t} \log P \left( t^{1-\varepsilon} <  |X(t)| < t^{1+\varepsilon} \right) \leq - (\alpha - (1+\alpha) \varepsilon).
\end{align*}
\item If $b=0$ and $1+\alpha<\beta\leq \gamma$, then for any $\varepsilon>0$, $\varepsilon<\alpha/\beta$
\begin{align*}
&- \alpha \leq \liminf_{n\to \infty} \frac{1}{\log t} \log P \left( t^{1-\varepsilon} <  |X(t)| < t^{1+\varepsilon} \right)\\
&\hspace{3cm} \leq \limsup_{t\to \infty} \frac{1}{\log t} \log P \left( t^{1-\varepsilon} <  |X(t)| < t^{1+\varepsilon} \right) \leq - (\alpha - \beta \varepsilon).
\end{align*}
\end{enumerate}
\end{theorem}

\begin{proof}
We prove only \textit{(i)}, the proof of \textit{(ii)} is similar. Under the assumptions in \textit{(i)}, the scaling function is (see \cite{GLT2019MomInfVar})
\begin{equation*}
	\tau(q) = \begin{cases}
	\frac{1}{1+\alpha} q, & 0<q\leq 1+\alpha,\\
	q-\alpha, & 1+\alpha \leq q < \gamma.
	\end{cases}
\end{equation*}
For the Legendre transform $\tau^*$ we have:
\begin{align*}
\tau^*(x) &= \max \left\{ \sup_{q <0} \left\{ qx - \tau(q) \right\}, \ \sup_{0\leq q \leq 1+\alpha} \left\{ q \left(x - \frac{1}{1+\alpha}\right) \right\}, \ \sup_{1+\alpha<q<\gamma} \left\{ q \left(x - 1\right) + \alpha \right\} \right\}\\
&=\begin{cases}
\max \left\{\sup_{q <0} \left\{ qx - \tau(q) \right\}, \ 0, \ (1+\alpha)x-1 \right\}, & \text{ if } x < \frac{1}{1+\alpha},\\
\max \left\{\sup_{q <0} \left\{ qx - \tau(q) \right\}, \ (1+\alpha)x-1, \ (1+\alpha)x-1 \right\}, & \text{ if } \frac{1}{1+\alpha} \leq x \leq 1,\\
\max \left\{\sup_{q <0} \left\{ qx - \tau(q) \right\}, \ (1+\alpha)x-1, \ \gamma(x-1)+\alpha \right\}, & \text{ if } x > 1.
\end{cases}
\end{align*}
As in the proof of Thm.~\ref{thm:supOUfv}, we use Prop.~\ref{propertiesoftau}\textit{(iv)} to get that for $q<0$
\begin{equation*}
\tau(q) \geq q \min \left\{ \inf_{0< q'\leq 1+\alpha} \frac{1}{1+\alpha}, \ \inf_{1+\alpha \leq q'< \gamma} \left(1- \frac{\alpha}{q'}\right) \right\} =  \frac{1}{1+\alpha} q.
\end{equation*}
Since
\begin{equation*}
\sup_{q <0} \left\{ qx - \tau(q) \right\} \leq \sup_{q <0} \left\{ qx - \frac{1}{1+\alpha} q \right\}=\begin{cases}
\infty, & \text{ if } x < \frac{1}{1+\alpha},\\
0, & \text{ if } x \geq \frac{1}{1+\alpha},
\end{cases}
\end{equation*}
we have
\begin{align*}
\tau^*(x) &= \begin{cases}
\max \left\{\sup_{q <0} \left\{ qx - \tau(q) \right\}, \ 0 \right\}, & \text{ if } x < \frac{1}{1+\alpha},\\
(1+\alpha)x-1, & \text{ if } \frac{1}{1+\alpha} \leq x \leq 1,\\
\gamma(x-1)+\alpha, & \text{ if } x > 1,
\end{cases}
\end{align*}
with $\{\frac{1}{1+\alpha}, 1\} \subset E$. Now we get the statement from \eqref{e:LDPsetA} by taking $A=\left(1- \varepsilon, 1+ \varepsilon\right)$.
\end{proof}

Note that in Thm.~\ref{thm:supOUinfv} we are not able to show that the rates greater that $1$ do not appear as it was shown in \eqref{e:thm51ii} of Thm.~\ref{thm:supOUfv}. This is due to infinite moments of order beyond $\gamma$. For the other combination of parameters not covered by Thm.~\ref{thm:supOUinfv}, the change-point in the shape of the scaling function does not appear in the range of finite moments (see \cite{GLT2019MomInfVar} for details). However, a finer approach could be used to show multiscale behavior by decomposing the integrated process into independent components as in \cite{GLT2019MomInfVar}.

\subsection{Gaussian case}
Suppose $Y$ is a Gaussian supOU process, that is $b>0$ and $\mu\equiv 0$. Assume in addition that $\E Y=0$ and that $\pi$ has a density $p$ satisfying \eqref{regvarofp} for some $\alpha>0$ and some slowly varying function $\ell$. Then the integrated process $X$ is also Gaussian and
\begin{equation*}
\left\{ \frac{1}{T^{1-\alpha/2 } \ell(T)^{1/2}} X(Tt) \right\} \overset{w}{\to} \left\{\widetilde{\sigma} B_H(t) \right\},
\end{equation*}
where $\{ B_H(t)\}$ is standard fractional Brownian motion with self-similarity parameter $H=1-\alpha/2$, $\widetilde{\sigma}^2 = b \Gamma(1+\alpha)/ ((2-\alpha)(1-\alpha))$ and the convergence is weak convergence in $C[0,1]$ (see \cite[Thms.~3.1 and 3.5]{GLT2019Limit}). Moreover, there is no intermittency and $\tau(q)=\left(1-\frac{\alpha}{2}\right) q$ for $q\geq 0$ \cite[Thm.~4.1]{GLT2019Limit}.

This case corresponds to the scenario of Subsec.~\ref{subsec:32} and we get that for any $\varepsilon>0$
\begin{equation*}
\lim_{t\to \infty} \frac{1}{\log t} \log P \left( |X(t)| > t^{1-\frac{\alpha}{2}+\varepsilon} \right) = -\infty,
\end{equation*}
Hence, the probability of $X(t)$ being larger than $t^{1-\frac{\alpha}{2}+\varepsilon}$ decays faster than any negative power of $t$. Since $X$ is Gaussian, finer estimates may be easily obtained, namely the classical large deviation principle and the moderate deviations given in the next theorem.

\begin{theorem}\label{thm:LDPGaussiansupOU}
Suppose $Y$ is a Gaussian supOU process such that $\E Y=0$, $\pi$ has a density $p$ satisfying \eqref{regvarofp} for some $\alpha \in (0,1)$ and some slowly varying function $\ell$ and let $X$ be the integrated process. For any sequence $\{s_t\}$ of positive numbers, $s_t\to \infty$, the process
\begin{equation*}
\frac{1}{\sqrt{s_t}} \frac{1}{t^{1-\alpha/2 } \ell(t)^{1/2}} X(t), \quad t >0,
\end{equation*}
satisfies the large deviation principle with speed $s_t$ and good rate function $\Lambda^*(x)= \frac{1}{2b} \frac{(2-\alpha)(1-\alpha)}{\Gamma(1+\alpha)} x^2$.
\end{theorem}

\begin{proof}
From \cite[Eq.~(5.3)]{GLT2019Limit} we have that
\begin{equation*}
\psi(\theta) := \log \E \left[ e^{\theta X(t)} \right] = \frac{b}{2} \theta^2  \int_0^\infty \int_0^{t}\left( 1 -e^{-\xi (t - s)} \right) ds  \xi^{-1} \pi(d\xi).
\end{equation*}
We now apply G\"artner-Ellis theorem \cite[Thm.~2.3.6]{dembo1998large} on the family $R(t) =  X(t) / \left( \sqrt{s_t} t^{1-\alpha/2 } \ell(t)^{1/2} \right)$. Then \eqref{e:lambdan} equals
\begin{equation*}
\Lambda_t(\theta) = \frac{b}{2} a_t^{-1} t^{-2+\alpha} \ell(t)^{-1} \theta^2  \int_0^\infty \int_0^{t}\left( 1 -e^{-\xi (t - s)} \right) ds  \xi^{-1} \pi(d\xi)
\end{equation*}
and we have by using \cite[Eqs.~(5.6) and (5.8)]{GLT2019Limit}
\begin{align*}
\frac{1}{a_t} \Lambda_t(t \theta)&= \frac{b}{2} t^{-2+\alpha} \ell(t)^{-1} \theta^2  \int_0^\infty \int_0^{t}\left( 1 -e^{-\xi (t - s)} \right) ds  \xi^{-1} \pi(d\xi)\\
&= \frac{b}{2} t^{-2+\alpha} \ell(t)^{-1} \theta^2 \int_0^{\infty} \left( 1 -e^{-w} \right) \int_{w/t}^{\infty} \xi^{-2} \pi(d\xi) dw\\
& \sim \frac{b}{2} t^{-2+\alpha} \ell(t)^{-1} \theta^2 \frac{\Gamma(1+\alpha)}{(2-\alpha)(1-\alpha)} \ell(t)  t^{2-\alpha} \\
& \sim \frac{b}{2} \frac{\Gamma(1+\alpha)}{(2-\alpha)(1-\alpha)} \theta^2 =:\Lambda(\theta).
\end{align*}
Since $\Lambda$ is essentially smooth and lower semicontinuous, we get the statement.
\end{proof}

For $s_t=t$ Thm.~\ref{thm:LDPGaussiansupOU} gives the classical large deviations in the Gaussian case. If we take $s_t=t^\varepsilon$ for $\varepsilon>0$, then Thm.~\ref{thm:LDPGaussiansupOU} shows that for any Borel set $A\subset \R$
\begin{equation*}
\begin{aligned}
- \inf_{x \in \Int(A) } \Lambda^* (x) &\leq \liminf_{t\to \infty} \frac{1}{t^\varepsilon} \log P \left( \frac{1}{t^{\frac{\varepsilon}{2}}} \frac{1}{t^{1-\alpha/2 } \ell(t)^{1/2}} X(t) \in A \right)\\
&\leq \limsup_{t\to \infty} \frac{1}{t^\varepsilon} \log P \left( \frac{1}{t^{\frac{\varepsilon}{2}}} \frac{1}{t^{1-\alpha/2 } \ell(t)^{1/2}} X(t) \in A \right) \leq - \inf_{x \in \cl (A)} \Lambda^* (x).
\end{aligned}
\end{equation*}
In particular, by taking $A=(M,\infty)$ for some $M>0$ we get that for any $\varepsilon>0$ the probability of large deviation
\begin{equation*}
P \left( \frac{1}{t^{1-\alpha/2 } \ell(t)^{1/2}} X(t) > M t^{\frac{\varepsilon}{2}}\right)
\end{equation*}
decays to zero as
\begin{equation*}
\exp \left\{ -\frac{1}{2b} \frac{(2-\alpha)(1-\alpha)}{\Gamma(1+\alpha)} M^2 t^\varepsilon \right\}, \quad \text{ as } t\to \infty.
\end{equation*}
This contrasts with the intermittent case, e.g.~the case of Thm.~\ref{thm:supOUfv}, where such probabilities decay as a power function of $t$. Hence, the classical large deviation principle with exponentially decaying probabilities does not hold for the supOU processes with intermittency and, in particular, the results of \cite{macci2017asymptotic} can not be extended to the intermittent case.

\section{Simulation}\label{sec:simulations}

We shall illustrate the multiscale behavior by a simple numerical example. Let $B_H$ and $B_b$ be independent fractional Brownian motions with Hurst parameters $H$ and $b$, respectively. We generate the values of the process $X$ at time points $t_n=n \Delta$, $n=1,\dots, T/\Delta$, where $T,\Delta >0$, by putting
\begin{equation*}
X(t_n) = \begin{cases}
B_H(t_n), & \text{ with probability } 1-t_n^{-a},\\
B_b(t_n), & \text{ with probability } t_n^{-a}.
\end{cases}
\end{equation*}
More precisely, we put
\begin{equation}\label{e:simXdef}
X(t_n) = \begin{cases}
B_H(t_n), & \text{ if } U_n=0,\\
B_b(t_n), & \text{ if } U_n=1,
\end{cases}
\end{equation}
where $U_n$, $n=1,\dots, T/\Delta$, are independent, $P(U_n=1)=1-P(U_n=0)=t_n^{-a}$, and independent of $B_H$ and $B_b$. The scaling function of $X$ is the same as in the biscale example of Subsec.~\ref{subsec:33} and is given by \eqref{e:biscaletau} for $q>-1$.

For the figures below we set $T=1000000$, $\Delta=1$, $H=0.6$ and $b=0.8$. Figure \ref{fig:path} shows three simulated sample paths of $X$ for two values of parameter $a$. The multiscaling behavior manifests as bursts along the sample path and these peaks have a magnitude much larger that the typical values of the sample path. The figures for both values, $a=0.8$ and $a=0.6$, are generated using the same sample paths of $B_H$ and $B_b$. One can notice here how lower value of $a$ makes the peaks more frequent. Figure \ref{fig:RoG} plots the rate of growth \eqref{e:RY} for the sample paths shown in Fig.~\ref{fig:path}. We can see here that the rate of growth is typically around $H=0.6$, but the burst also illustrate that the rate $b=0.8$ also appears.

%
%\begin{figure}[h!]
%\centering
%\begin{subfigure}[b]{0.47\textwidth}
%\centering
%\includegraphics[width=\textwidth]{Figpatha}
%\caption{$a=0.8$}
%\label{fig:path:a}
%\end{subfigure}
%\hfill
%\begin{subfigure}[b]{0.47\textwidth}
%\centering
%\includegraphics[width=\textwidth]{Figpathb}
%\caption{$a=0.6$}
%\label{fig:path:b}
%\end{subfigure}
%\caption{Simulated sample paths of $X$ given by \eqref{e:simXdef}}
%\label{fig:path}
%\end{figure}
%
%
%
%
%\begin{figure}[h!]
%\centering
%\begin{subfigure}[b]{0.47\textwidth}
%\centering
%\includegraphics[width=\textwidth]{FigRoGa}
%\caption{$a=0.8$}
%\label{fig:RoG:a}
%\end{subfigure}
%\hfill
%\begin{subfigure}[b]{0.47\textwidth}
%\centering
%\includegraphics[width=\textwidth]{FigRoGb}
%\caption{$a=0.6$}
%\label{fig:RoG:b}
%\end{subfigure}
%\caption{The rate of growth \eqref{e:RY} of simulated paths of $X$ given by \eqref{e:simXdef}}
%\label{fig:RoG}
%\end{figure}

%%%%%%%%%%%%%%%%%%%%%%%
% SMALL png FIGURES FOR SENDING
%%%%%%%%%%%%%%%%%%%%%%

\begin{figure}[h!]
\centering
\begin{subfigure}[b]{0.47\textwidth}
\centering
\includegraphics[width=\textwidth]{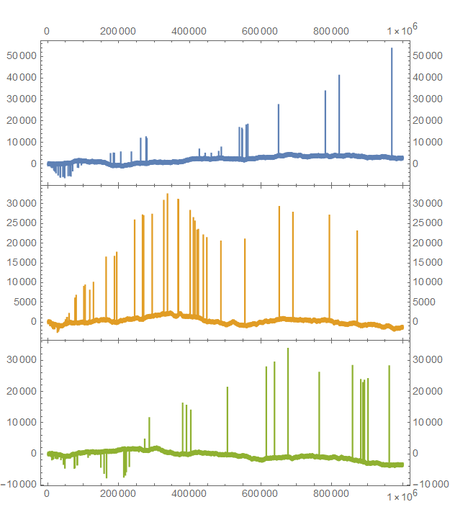}
\caption{$a=0.8$}
\label{fig:path:a}
\end{subfigure}
\hfill
\begin{subfigure}[b]{0.47\textwidth}
\centering
\includegraphics[width=\textwidth]{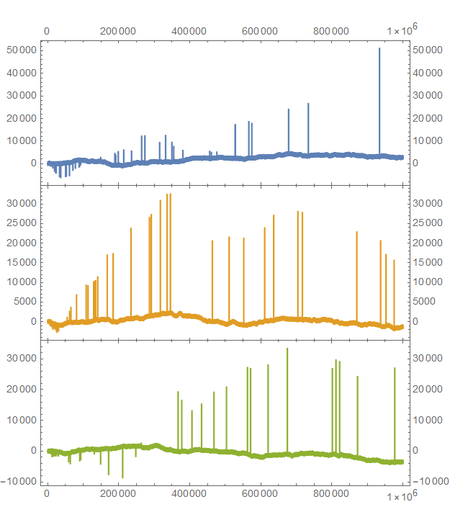}
\caption{$a=0.6$}
\label{fig:path:b}
\end{subfigure}
\caption{Simulated sample paths of $X$ given by \eqref{e:simXdef}}
\label{fig:path}
\end{figure}

\begin{figure}[h!]
\centering
\begin{subfigure}[b]{0.47\textwidth}
\centering
\includegraphics[width=\textwidth]{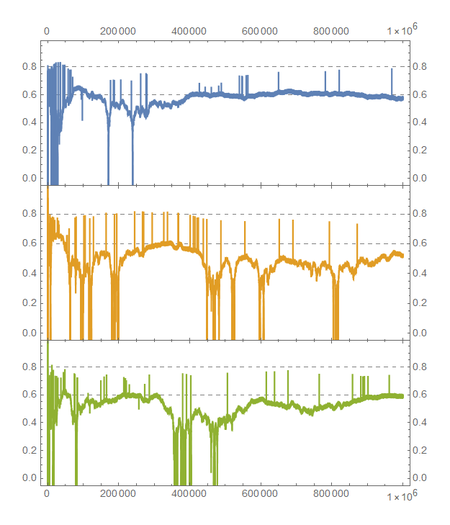}
\caption{$a=0.8$}
\label{fig:RoG:a}
\end{subfigure}
\hfill
\begin{subfigure}[b]{0.47\textwidth}
\centering
\includegraphics[width=\textwidth]{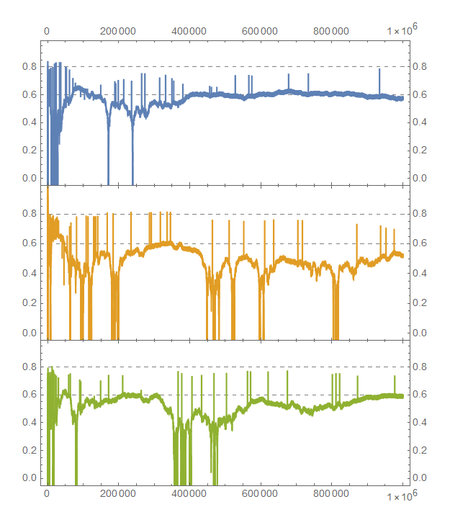}
\caption{$a=0.6$}
\label{fig:RoG:b}
\end{subfigure}
\caption{The rate of growth \eqref{e:RY} of simulated paths of $X$ given by \eqref{e:simXdef}}
\label{fig:RoG}
\end{figure}

\section{Conclusion and discussion}\label{sec6}
The technique used in this paper to show multiscale behavior is very general and is based on the rate of growth of moments which is closely related to the phenomenon of intermittency. However, this approach has some limitations. It is efficient in proving that some rates are significant on the power probability scale, but it may not show that some rates are negligible when they really are (see the example in Subsec.~\ref{subsec:33}). These problems also appear when considering infinite moments. In general, finite moments of positive order help in showing rates beyond some critical rate do not appear (the process never grows faster than some critical rate). On the other hand, finite moments of negative order help showing that rates less than some critical rate do not appear (the process never grows slower than some critical rate). These points are illustrated in Subsecs.~\ref{subsec:31} and \ref{subsec:32}, but also in Sec.~\ref{subsec:34} on the supOU processes.

The scaling function, however, cannot be used to reveal all the scales that the process exhibits. To see this, consider the following extension of the example from Subsec.~\ref{subsec:33}. Suppose that $X(t)$, $t \in \N$, is a sequence given by
\begin{equation*}
X(t) = \begin{cases}
t^{H}, & \text{ with probability } 1-t^{-a/2}-t^{-a},\\
t^{(H+b)/2}, & \text{ with probability } t^{-a/2},\\
t^b, & \text{ with probability } t^{-a},
\end{cases}
\end{equation*}
where $0<H<b$ and $a>0$. The scaling function for $q\in \R$ is the same as in the example from Subsec.~\ref{subsec:33}, namely
\begin{align*}
\tau(q) &= \lim_{t\to \infty} \frac{1}{\log t} \log \left( t^{H q} \left(1-t^{-a/2}-t^{-a} \right) + t^{(H+b)q/2-a/2} +  t^{bq-a} \right)\\
&=\begin{cases}
H q, & \text{ if } q\leq \frac{a}{b-H},\\
bq-a, & \text{ if } q> \frac{a}{b-H}.
\end{cases}
\end{align*}
Hence, just from the form of the scaling function one is not able to reveal that $X(t)$ also exhibits the intermediate scale $t^{(H+b)/2}$. This is to be expected since the scaling function only focuses on the behavior of the moments. It is nevertheless a useful tool.

\bigskip
\bigskip

\textbf{Acknowledgements}
Nikolai N.~Leonenko was supported in particular by Cardiff Incoming Visiting Fellowship Scheme, International Collaboration Seedcorn Fund, Australian Research Council’s Discovery Projects funding scheme (project DP160101366) and the project MTM2015-71839-P of MINECO, Spain (co-funded with FEDER funds). Murad S.~Taqqu was supported in part by the Simons foundation grant 569118 at Boston University. Danijel Grahovac was partially supported by the University of Osijek Grant ZUP2018-31.

%Nikolai N.~Leonenko was supported in particular by Cardiff Incoming Visiting Fellowship Scheme, International Collaboration Seedcorn Fund, Australian Research Council's Discovery Projects funding scheme (project DP160101366)and the project MTM2015-71839-P of MINECO, Spain (co-funded with FEDER funds). Murad S.~Taqqu was supported in part by the Simons foundation grant 569118 at Boston University.
%D.~Grahovac acknowledges the support of University of Osijek grant ZUP2018-31.

\bigskip
\bigskip
\bibliographystyle{abbrv}
\bibliography{References}

\end{document}